\documentclass[12pt]{amsart}
\usepackage[inner=1in,outer=1in,top=1in,bottom=1in]{geometry}
\usepackage[utf8]{inputenc}
\usepackage[T1]{fontenc}
\usepackage{url,  mathrsfs}
\usepackage{amsmath}
\usepackage{amssymb}
\usepackage{amsthm}
\usepackage{graphicx, epstopdf}
\usepackage{subfig}
\usepackage{color}
\usepackage{bbm} 
\usepackage{dsfont}
\usepackage[shortlabels]{enumitem}
\usepackage{subfloat}
\usepackage[draft]{fixme}
\usepackage{bm}
\usepackage{bbm} 
\usepackage{ifpdf}
\usepackage{array}
\usepackage{mathtools}
\usepackage{multicol}
\usepackage{multirow}
\usepackage{graphicx}
\usepackage{tikz}
\usepackage{tikz-3dplot}
\usepackage{hyperref}
\usepackage[capitalize]{cleveref}

\newcommand{\C}{\mathbb{C}}
\newcommand{\R}{\mathbb{R}}
\newcommand{\Q}{\mathbb{Q}}
\newcommand{\Z}{\mathbb{Z}}
\newcommand{\kk}{\mathbbm{k}}

\newcommand\ps{\{\!\{t\}\!\}}
\newcommand{\cS}{\mathcal{S}}
\newcommand{\cJ}{\mathcal{J}}

\renewcommand{\subset}{\subseteq}
\renewcommand{\emptyset}{\varnothing}

\DeclareMathOperator{\conv}{conv}

\DeclareMathOperator{\supp}{supp}

\DeclareMathOperator{\dom}{dom}
\DeclareMathOperator{\Dr}{Dr}

\DeclareMathOperator{\chara}{char}

\newtheorem{theorem}{Theorem}[section]
\newtheorem*{theorem*}{Theorem}
  \newtheorem{lemma}[theorem]{Lemma}

   \newtheorem{corollary}[theorem]{Corollary}
  \newtheorem{proposition}[theorem]{Proposition}

\newtheorem{conjecture}[theorem]{Conjecture}

 \theoremstyle{definition}

  \newtheorem{remark}[theorem]{Remark}
\newtheorem{example}[theorem]{Example}
\newtheorem*{example*}{Example}

\usepackage[draft]{fixme}
\fxsetup{theme=color, mode=multiuser, noinline}
\FXRegisterAuthor{gl}{GL}{\color{red}Gaku}
\FXRegisterAuthor{cv}{CV}{\color{blue}Cynthia}

\title{Valuated Delta Matroids and Principal Minors of Hermitian matrices}
\author{Nathan Cheung}
\email{ncheung@uw.edu}
\author{Tracy Chin}
\email{tlchin@uw.edu}
\author{Gaku Liu}
\email{gakuliu@uw.edu}
\author{Cynthia Vinzant}
\email{vinzant@uw.edu}
\address{University of Washington, Seattle, USA} 

\begin{document}

\begin{abstract}
    In this paper we introduce valuated $\Delta$-matroids, a natural generalization of two objects of study in matroid theory: valuated matroids and $\Delta$-matroids. We show that these objects exhibit nice properties analogous to ordinary valuated matroids. We also show that these objects arise as the valuations of principal minors of a Hermitian matrix over a valued field, generalizing other forms of $\Delta$-matroid representability.
\end{abstract}

\maketitle

Valuated matroids are a well-studied generalization of matroids. Matroids can be thought of as functions $\binom{[n]}{r} \to  \{0,\infty\}$ satisfying a set of ``combinatorial Pl\"{u}cker relations'', which essentially reduce to the basis exchange axiom.
\emph{Valuated matroids} were defined by Dress and Wenzel \cite{DW92} as a combinatorial abstraction of matroids represented over a valued field.  They are functions $\binom{[n]}{r} \to  \R \cup \{\infty\}$ satisfying a set of relations called the \emph{tropical Pl\"{u}cker relations}. For example, if $K$ is a field equipped with a nonarchimedean valuation $\nu : K \to \R \cup \{\infty\}$ and $v_1,v_2,\dots,v_n$ are vectors which span $K^r$, then the map
\[
\{i_1 < i_2 < \dots < i_r\} \mapsto \nu( \det (v_{i_1}, \dots , v_{i_r}) )
\]
is a valuated matroid. Such valuated matroids are called \emph{representable}. Valuated matroids specialize to ordinary matroids by restricting the codomain to $\{0,\infty\}$.

Delta matroids, or $\Delta$-matroids, can be thought of as the type $B$ analogue of matroids (which are ``type $A$''). A \emph{$\Delta$-matroid} is a family of subsets of $[n]$, called bases, such that if $A$, $B$ are bases and $a \in A \Delta B$, then there exists $b \in A \Delta B$ such that $A \Delta \{a,b\}$ is a basis. If in the definition we further require that $a \neq b$, then the family is an \emph{even $\Delta$-matroid}.

A common and important way to study matroids is through their polytopes. The \emph{matroid polytope} of a matroid is the convex hull of the indicator vectors of the bases of the matroid. A 0-1 polytope is the matroid polytope of a matroid if and only if all of its edges are in the direction $e_i - e_j$ for some $i$, $j$. Analogously, the \emph{$\Delta$-matroid polytope} of a $\Delta$-matroid is the convex hull of the indicator vectors of its bases. A 0-1 polytope is the $\Delta$-matroid polytope of a $\Delta$-matroid if and only if all its edges are in the direction $e_i$, $e_i - e_j$, or $e_i + e_j$ for some $i$, $j$. A 0-1 polytope is the $\Delta$-matroid polytope of an \emph{even} $\Delta$-matroid if and only if all its edges are in the direction $e_i - e_j$ or $e_i + e_j$ for some $i$, $j$.

Given a finite set of points $X \subset \R^n$ and a function $p : X \to \R \cup \{\infty\}$, we obtain a polytopal complex in $\R^n$ by projecting the lower faces of the polytope $\conv\{ (x,p(x)) : x \in X, p(x) \neq \infty \}$. This complex is called the \emph{subdivision induced by $p$}, and is a polyhedral subdivision of some polytope with vertices in $X$. A function $\binom{[n]}{r} \to  \R \cup \{\infty\}$ is equivalent to a function $\Delta_{n,r} \to \R \cup \{\infty\}$ where $\Delta_{n,r}$ is the set of all 0-1 vectors in $\R^n$ with $r$ 1's, and thus induces a subdivision in $\R^n$. Such a function is a \emph{valuated matroid} if and only if every polytope in the induced subdivision is a matroid polytope.

It is natural to try to give a common generalization of these notions. Wenzel \cite{Wenzel} defines a \emph{strongly valuated $\Delta$-matroid} to be a function $\{0,1\}^n \to \R \cup \{\infty\}$ which satisfies a set of relations called the \emph{tropical Wick relations}. For example, given a skew symmetric matrix over a valued field, the valuations of the principal minors form a strongly valuated $\Delta$-matroid.
Rinc\'{o}n \cite{Felipe} shows that a function $\{0,1\}^n \to \R \cup \{\infty\}$ is a strongly valuated $\Delta$-matroid if and only if every polytope in the induced subdivision is an \emph{even} $\Delta$-matroid polytope. Note that Rinc\'{o}n, as well as other authors, refer to Wenzel's strongly valuated $\Delta$-matroids as just valuated $\Delta$-matroids.

\begin{example*}
    Consider the $3\times 3$ symmetric matrix 
    $\begin{pmatrix} 0 & 1 & 1 \\ 1 & 0 & 1 \\ 1 & 1& 0\end{pmatrix}$ over a field of characteristic $\neq 2$. The principal minors of size $1$ are zero and all others are nonzero. 
    The corresponding $\Delta$-matroid has bases $\emptyset, \{1,2\}, \{1,3\}, \{2,3\}$, and $\{1,2,3\}$, which is not even. Indeed, 
    taking $A = \{1,2,3\}$, $B = \emptyset$, and any $a\in A$, we see that this fails to satisfy the \emph{strong} exchange axiom, in which there must exist $b$ so that both $A\Delta\{a,b\}$ and $B\Delta\{a,b\}$ are bases. In this example, the only choice of $b$ for which  $A\Delta\{a,b\}$ is a basis is $b=a$, in which case $A\Delta\{a,b\} = \{a\}$ is not.  To study such representable $\Delta$-matroids and their valuated analogues, we therefore give up the strong exchange axiom. 
\end{example*}

In this paper, we call a function $p : \{0,1\}^n \to \R \cup \{\infty\}$ a \emph{valuated $\Delta$-matroid} if every polytope in the induced subdivision is a $\Delta$-matroid polytope (not necessarily even). In other words, every edge of the induced regular subdivision has direction $\pm e_i$ or $e_i \pm e_j$ for some $i,j\in [n]$. We note that this conflicts with other recent terminology as mentioned above. However, we believe our definition is the natural definition of a valuated $\Delta$-matroid which is consistent with the history and other terminology of $\Delta$-matroids. As we demonstrate in this paper, these objects are interesting in their own right and generalize previous concepts in the study of $\Delta$-matroids, including non-even $\Delta$-matroids.

A summary of our results is as follows:

\begin{itemize}[left=0pt .. \parindent]
    \item As mentioned above, ordinary valuated matroids can be characterized as functions $\binom{[n]}{r} \to  \R \cup \{\infty\}$ satisfying a set of local relations called the tropical Pl\"{u}cker relations. Geometrically, this is equivalent to saying a function is a valuated matroid if and only if its induced subdivision has no edges of length $3$, where the \emph{length} of a vector $(x_1,\dots,x_n)$ is defined to be its $L_1$-length $\sum_{i=1}^n |x_i|$.

    In analogy to this, we prove that a function $\{0,1\}^n \to \R \cup \{\infty\}$ is a valuated $\Delta$-matroid if and only if its induced subdivision has no edges of length $3$ or $4$. This means that determining whether a function is a valuated $\Delta$-matroid can be done by checking a list of local conditions which is bounded up to symmetry, as is the case with ordinary valuated matroids. The exact conditions are given in Lemma~\ref{lem:valuated-delta-matroid-3} and Theorem~\ref{thm:n4-charac}.

    \item We discuss the rank function, and use it to prove that any $\Delta$-matroid can appear as the cell of a $\Delta$-matroid subdivision.

    \item The set of all functions $\binom{[n]}{r} \to  \R \cup \{\infty\}$ that are valuated matroids form a polyhedral fan called the \emph{Dressian}. Analogously, we call the set of all functions $\{0,1\}^n \to \R \cup \{\infty\}$ that are valuated $\Delta$-matroids the \emph{$\Delta$-Dressian}, $\Delta$-${\rm Dr}(n)$. (This conflicts slightly with the usage in \cite{Felipe}.) In \Cref{thm:DimDrBounds}, we establish bounds on the dimension of this object, showing that is exponential in $n$.

    \item Given a symmetric or skew-symmetric matrix over any field \cite{Bouchet} or a Hermitian matrix in $\C^{n \times n}$ \cite{B07}, the nonzero principal minors of the matrix form a $\Delta$-matroid. We generalize this result to show that given any valued field equipped with an automorphic involution preserving the valuation, and any matrix over this field which is Hermitian or skew-Hermitian with respect to the involution, the valuations of the principal minors of the matrix form a valuated $\Delta$-matroid.

    We show that in dimension at most three, all valuated $\Delta$-matroids can be realized in this way up to symmetry. We show this is not true in dimension four. Finally, we conjecture that if a matrix is the sum of a skew-Hermitian matrix and a rank one Hermitian matrix, then the valuations of its principal minors form a $\Delta$-matroid. We prove this is several cases, including the case of a skew-symmetric matrix plus a rank one symmetric matrix.

    \item We show how our results on Hermitian matrices can be rewritten in terms of sesquilinear forms and isotropic subspaces, and that this generalizes previous ways of representing $\Delta$-matroids through these subspaces.
\end{itemize}

{\bf Acknowledgements.} We would like to thank Matt Baker, Changxin
Ding, Alex Fink, Donggyu Kim, Felipe Rinc\'on, and Josephine Yu for helpful discussions and comments.
	Authors Chin and Vinzant were partially supported by NSF grant DMS-2153746. Author Liu was supported by NSF grant DMS-2348785.

\section{$\Delta$-matroid polytopes and regular subdivisions}

A \emph{$\Delta$-matroid} on $[n]$ is a family of subsets of $[n]$, called bases, such that if $A$, $B$ are bases and $a \in A \Delta B$, then there exists $b \in A \Delta B$ such that $A \Delta \{a,b\}$ is a basis. To each subset $S \subseteq [n]$ we associate its \emph{indicator vector} $e_S := \sum_{i \in S} e_i \in \R^n$, where $e_i$ is the $i$-th standard basis vector of $\R^n$. The convex hull of the indicator vectors of a $\Delta$-matroid is called a \emph{$\Delta$-matroid polytope}. A polytope in $\R^n$ with vertices in $\{0,1\}^n$ is a $\Delta$-matroid polytope if and only if all its edges are in the direction $\pm e_i$ or $\pm e_i \pm e_j$ for some $i$, $j$. From now on, we will refer to $\Delta$-matroid polytopes as simply ``$\Delta$-matroids'' and conflate the following objects: a $\Delta$-matroid, its associated $\Delta$-matroid polytope, and the set of vertices of this polytope (i.e. the indicator vectors of the bases).

Given a line segment with endpoints in $\{0,1\}^n$, we define its \emph{length} to be the Hamming distance between its endpoints, i.e.\ the number of coordinates at which the endpoints differ. Hence, a 0-1 polytope is a $\Delta$-matroid if and only if all its edges have length at most 2.

Let $\R \cup \{\infty\}$ be the real number system extended by positive $\infty$. The only operations we will perform involving $\infty$ will be $+$, $\max$, $\min$, and multiplication by a positive scalar, in which case $\infty$ behaves as one would expect (e.g., $a + \infty = \infty$ for all $a \in \R \cup \{\infty\}$).

A \emph{polytopal complex} is a collection $\cS$ of polytopes in $\R^n$ satisfying the following properties.
\begin{enumerate}
    \item If $P \in \cS$ and $F$ is a face of $P$, then $F \in \cS$.
    \item If $P$, $Q \in \cS$ and $P \neq Q$, then the relative interiors of $P$ and $Q$ are disjoint.
\end{enumerate}
The elements of $\cS$ are called \emph{cells}.
The \emph{support} of $\cS$ is $\bigcup_{P \in \cS} P$. A polytopal complex whose support is a polytope $Q$ is called a \emph{polytopal subdivision}, or \emph{subdivision}, of $Q$.

Let $X \subset \R^n$ be a finite set of points and $p : X \to \R \cup \{\infty\}$ any function. The \emph{effective domain} of $p$ is $\dom p := \{x \in X : p(x) \neq \infty \}$. Define a polytope $P_p$ in $\R^{n+1}$ by
\[
P_p := \conv \{ (x,p(x)) : x \in \dom p \}.
\]
The projection $\R^{n+1} \to \R^n$ from deleting the last coordinate maps $P_p$ onto $\conv \dom p$. Moreover, it maps each lower face of $P_p$ (that is, each face of $P_p$ supported by a hyperplane whose inner normal has positive last coordinate) to a polytope contained in $\conv \dom p$. The collection of all projected lower faces forms a polytopal subdivision of $\conv \dom p$, which is called the subdivision \emph{induced} by $p$. We denote it by $\mathcal S_p$.

We next give a characterization of when a simplex appears in $\cS_p$. Recall that if $I$ is an affinely independent set of points and $x$ is in the relative interior of $\conv I$, then $x$ can be written uniquely as $x = \sum_{i \in I} \lambda_i i$ where $0 < \lambda_i < 1$ and $\sum_{i \in I} \lambda_i = 1$. We call this the \emph{convex circuit representation} of $x$ with support $I$. If the elements of $I$ have weights given by a function $p$, we define
\[
p(x, I) := \sum_{i \in I} \lambda_i p(i).
\]

\begin{proposition} \label{prop:cellcharacterization}
   Let $X \subset \R^n$ be a finite set of points in convex position, and let $p : X \to \R \cup \{\infty\}$ be any function. Let $I$ be an affinely independent subset of $X$. The following are equivalent.
    \begin{enumerate}
        \item $I$ is the set of vertices of a simplex in $\cS_p$.
        \item For any $x$ in the relative interior of $\conv I$, and for any convex circuit representation of $x$ with different support $J \subseteq X$, we have $p(x,I) < p(x,J)$.
        \item There exists $x$ in the relative interior of $\conv I$, such that for any convex circuit representation of $x$ with different support $J \subseteq X$, we have $p(x,I) < p(x,J)$.
    \end{enumerate}
\end{proposition}

\begin{proof}
    (1) $\implies$ (2). Suppose $I$ is the set of vertices of a simplex in $\cS_p$. Then $I_p := \{ (x,p(x)) : x \in I \}$ is the set of vertices of a lower face of $P_p$. Hence, there is a linear functional $(\phi,1) \in (\R^n)^\ast \times \R$ and real number $b$ such that $(\phi,1)(x,y) = b$ for all $(x,y) \in \conv(I_p) \subset \R^n \times \R$, and $(\phi,1)(x,y) > b$ for all $(x,y) \in P_p \setminus \conv(I_p)$. 
    
    Let $x$ be in the relative interior of $\conv I$, and let $\sum_{i \in I} \lambda_i i$ be its convex circuit representation with support $I$. We have
    \[
        \sum_{i \in I} \lambda_i (\phi,1)(i,p(i)) = b \implies \phi(x) + p(x,I) = b \implies p(x,I) = -\phi(x) + b.
    \]
    Suppose that $x$ is also in the relative interior of $\conv J$, where $J \subseteq X$ and $J \neq I$. Since $X$ is in convex position, at least one element of $J$ is not in $\conv I$. Let $\sum_{j \in J} \lambda_j' j$ be the convex support representation of $x$ with support $J$. Then
    \[
        \sum_{j \in J} \lambda_j' (\phi,1)(j,p(j)) > b \implies \phi(x) + p(x,J) > b \implies p(x,J) > -\phi(x) + b.
    \]
    It follows that $p(x,I) < p(x,J)$, as desired. (If any element of $J$ is not in $\dom p$, then this inequality is obvious as the right side is finite and the left side is $\infty$.)

    (2) $\implies$ (3). Immediate.

    (3) $\implies$ (1). We prove the contrapositive. Suppose $I$ is not the set of vertices of a simplex in $\cS_p$. Choose any $x$ in the relative interior of $\conv I$. Let $F \subset X$ be the set of vertices of the cell of $\cS_p$ whose relative interior contains $x$. Since $x$ is in the relative interior of $\conv I$, $F$ is not a proper subset of $I$. In addition, $F \neq I$ by assumption. Hence $F$ contains a point not in $I$. It follows that there exists $J \subset F$, $J \neq I$, such that $J$ is affinely independent and $x$ is in the relative interior of $\conv J$. (For example, one can consider a pulling triangulation of $\conv F$ starting with a vertex not in $I$.) Using the same argument as above, we have $p(x,J) \le p(x,I)$, and thus (3) is not true.
\end{proof}

We call a function $p : \{0,1\}^n \to \R \cup \{\infty\}$ a \emph{valuated $\Delta$-matroid} if every cell in the induced subdivision $\cS_p$ is a $\Delta$-matroid. Equivalently, a function $p : \{0,1\}^n \to \R \cup \{\infty\}$ is a valuated $\Delta$-matroid if and only if all the edges (1-dimensional cells) of its induced subdivision have length at most 2. Slightly more generally, given a subset $M$ of $\{0,1\}^n$ such that $M$ is a $\Delta$-matroid, a function $p : M \to \R \cup \{\infty\}$ is a \emph{valuated $\Delta$-matroid} if every cell in the induced subdivision $\cS_p$ is a $\Delta$-matroid. Crucially, if $p : \{0,1\}^n \to \R \cup \{\infty\}$ is a valuated $\Delta$-matroid, then the restriction of $p$ to any face of $[0,1]^n$ is also a valuated $\Delta$-matroid.

We set some final notation. Given a function $p : \{0,1\}^n \to \R \cup \{\infty\}$, we will often abbreviate the value $p(e_S)$ as $p_S$. In addition, we will often abbreviate a set $\{a_1,\dots,a_k\}$ as $a_1 \dots a_k$; for example $p(e_{\{1,2,3\}})$ would be written $p_{123}$. For a set $S \subseteq [n]$, we define $S^C := [n] \setminus S$ to be its complement in $[n]$.

\section{The cases $n=3$ and $n=4$}

Before we describe a characterization for valuated $\Delta$-matroids on $\{0,1\}^n$ for all $n$, we first work with $n=3$ and $4$. In \cref{sec:characterization-for-all-n}, we will show that our constraints for $n=3$ and $4$ are enough to characterize valuated $\Delta$-matroids for all $n$.

\begin{lemma}\label{lem:valuated-delta-matroid-3}
    A function $p : \{0,1\}^3 \to \R \cup \{\infty\}$ is a valuated $\Delta$-matroid on $[3]$ if and only if the minimum
    \begin{align*}
    \min(2p_\emptyset + 2p_{123}, 2p_1 + 2p_{23}, 2p_2 + 2p_{13}&, 2p_3 + 2p_{12}, \\
    &p_\emptyset + p_{12} + p_{13} + p_{23}, p_1 + p_2 + p_3 + p_{123})
    \end{align*}
    is achieved on at least two of its arguments, or on one of the last two arguments.
\end{lemma}

\begin{proof}
    We will show the contrapositive in both directions.

    First, suppose that $p : \{0,1\}^3 \to \R \cup \{\infty\}$ is not a valuated $\Delta$-matroid. Then the induced subdivision $\cS_p$ has an edge of length 3, so $[e_S,e_{S^C}]$ is an edge for some $S \subseteq [3]$. Without loss of generality, assume $[0,e_{123}]$ is an edge. Let $x = (1/2,1/2,1/2)$ be the midpoint of this edge. We list all of the convex circuit representations of $x$ using subsets of $\{0,1\}^3$:
    \begin{multline*}
    \frac{1}{2}(e_\emptyset + e_{123}), \frac{1}{2}(e_1 + e_{23}), \frac{1}{2}(e_2+e_{13}), \frac{1}{2}(e_3+e_{12}), \\
    \frac{1}{4}(e_\emptyset + e_{12} + e_{13} + e_{23}), \frac{1}{4}(e_1 + e_2 + e_3 + e_{123})
    \end{multline*}
    By Proposition~\ref{prop:cellcharacterization}, $[0,e_{123}]$ being an edge means that in the list of numbers
    \begin{multline*}
    \frac{1}{2}(p_\emptyset + p_{123}), \frac{1}{2}(p_1 + p_{23}), \frac{1}{2}(p_2+p_{13}), \frac{1}{2}(p_3+p_{12}), \\
    \frac{1}{4}(p_\emptyset + p_{12} + p_{13} + p_{23}), \frac{1}{4}(p_1 + p_2 + p_3 + p_{123})
    \end{multline*}
    the minimum is achieved uniquely at $\frac{1}{2}(p_\emptyset + p_{123})$. In particular, the minimum is not achieved twice nor at one of the last two arguments, as desired.

    Conversely, suppose the above minimum is not achieved on at least two arguments nor on one of the last two arguments. Then it must be achieved uniquely on one of the first four arguments, that is, $\frac{1}{2}(p_S + p_{S^C})$ for some $S \subseteq [3]$. By Proposition~\ref{prop:cellcharacterization}(3) applied to the point $x = (1/2,1/2,1/2)$, it follows that $[e_S,e_{S^C}]$ is an edge in $\cS_p$. Thus $\cS_p$ is not a $\Delta$-matroid.
\end{proof}

\begin{figure}
    \centering
\tdplotsetmaincoords{70}{130}
\begin{tikzpicture}[axis/.style={->,black,thick}, tdplot_main_coords]
    \def\R{2}

    \draw[canvas is xy plane at z=0, thick]
        (\R,0) -- (\R, \R) -- (0, \R) -- (0,0) -- cycle;
    \draw[canvas is xz plane at y=0, thick]
        (\R,0) -- (\R, \R) -- (0, \R) -- (0,0) -- cycle;
    \draw[canvas is yz plane at x=0, thick]
        (\R,0) -- (\R, \R) -- (0, \R) -- (0,0) -- cycle;

    \draw[fill=gray, opacity=0.5]
        (\R,0,0) -- (\R,\R,0) -- (0, \R, \R) -- (0, 0, \R) -- cycle;

    \draw[canvas is xy plane at z=\R, thick]
        (\R,0) -- (\R, \R) -- (0, \R) -- (0,0) -- cycle;
    \draw[canvas is xz plane at y=\R, thick]
        (\R,0) -- (\R, \R) -- (0, \R) -- (0,0) -- cycle;
    \draw[canvas is yz plane at x=\R, thick]
        (\R,0) -- (\R, \R) -- (0, \R) -- (0,0) -- cycle;

    \begin{scope}[shift={(4 cm, 0 cm)}]
        \draw[canvas is xy plane at z=0, thick]
            (\R,0) -- (\R, \R) -- (0, \R) -- (0,0) -- cycle;
        \draw[canvas is xz plane at y=0, thick]
            (\R,0) -- (\R, \R) -- (0, \R) -- (0,0) -- cycle;
        \draw[canvas is yz plane at x=0, thick]
            (\R,0) -- (\R, \R) -- (0, \R) -- (0,0) -- cycle;

        \draw[thick, fill=gray, opacity=0.5]
            (0,\R,0) -- (\R,0,0) -- (0, 0, \R) -- cycle;
        \draw[thick, fill=gray, opacity=0.5]
            (\R,0,0) -- (\R,\R,\R) -- (0, \R, 0) -- cycle;
        \draw[thick, fill=gray, opacity=0.5]
            (\R,0,0) -- (\R,\R,\R) -- (0, 0, \R) -- cycle;
        \draw[thick, fill=gray, opacity=0.5]
            (0,\R,0) -- (\R,\R,\R) -- (0, 0, \R) -- cycle;

        \draw[canvas is xy plane at z=\R, thick]
            (\R,0) -- (\R, \R) -- (0, \R) -- (0,0) -- cycle;
        \draw[canvas is xz plane at y=\R, thick]
            (\R,0) -- (\R, \R) -- (0, \R) -- (0,0) -- cycle;
        \draw[canvas is yz plane at x=\R, thick]
            (\R,0) -- (\R, \R) -- (0, \R) -- (0,0) -- cycle;
    \end{scope}
\end{tikzpicture}
    \caption{Illustration of two cases in \cref{lem:valuated-delta-matroid-3}. The left cube is the case where the minimum is achieved on two of the first four arguments, and the right cube is the case where the minimum is achieved uniquely on one of the last two arguments.}
    \label{fig:subdivisions-3}
\end{figure}
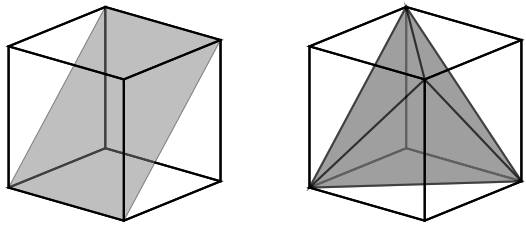

Before stating the characterization for $n=4$, we first need to describe the convex circuit representations of $(1/2,1/2,1/2,1/2)$ with support in $\{0,1\}^4$. Let $B_n$ denote the symmetry group of the $n$-cube. It is generated by permutations of the coordinates and flipping the bits of coordinates.
The following can be verified computationally:
\begin{proposition} \label{prop:n4-circuits}
    The convex circuit representations of $(1/2,1/2,1/2,1/2)$ are exactly the images of the following formal linear combinations under $B_4$:
    \[
    \frac{1}{2}(e_\emptyset + e_{1234}),\ \frac{1}{4}(e_\emptyset + e_{14} + e_{123} + e_{234}),\
    \frac{1}{3}e_\emptyset + \frac{1}{6}(e_{123} + e_{124} + e_{134} + e_{234}).
    \]
\end{proposition}

Let $\cJ$ be the set of all supports of convex circuit representations from Proposition~\ref{prop:n4-circuits}.
We now state the result for $n = 4$.

\begin{theorem}\label{thm:n4-charac}
A function $p : \{0,1\}^4 \to \R \cup \{\infty\}$ is a valuated $\Delta$-matroid on $[4]$ if and only if both of the following hold:
\begin{enumerate}
    \item $p$ restricts to a valuated $\Delta$-matroid on each facet of $[0,1]^4$.
    \item For $x = (1/2,1/2,1/2,1/2)$, the minimum
    \[
    \min_{J \in \cJ} p(x, J)
    \]
    is achieved on at least two of its arguments.
\end{enumerate}
\end{theorem}

\begin{proof}
    Let $\cJ = \cJ_1 \sqcup \cJ_2 \sqcup \cJ_2$, where $\cJ_1$, $\cJ_2$, $\cJ_3$ are the orbits of $\{\emptyset,1234\}$, $\{\emptyset,14,123,234\}$, and $\{\emptyset,123,124,134,234\}$, respectively, under the action of $B_4$. We again prove the contrapositive in both directions.
    
    Suppose $p : \{0,1\}^4 \to \R \cup \{\infty\}$ is not a valuated $\Delta$-matroid. Then its induced subdivision $\cS_p$ either it has an edge of length 3 or an edge of length 4. If it has an edge of length 3, then this edge is contained in some facet of $[0,1]^4$. Therefore (1) is not true, as desired. Suppose $\cS_p$ has an edge of length 4. The set of endpoints of this edge is an element $J \in \cJ_1$. Therefore, by Proposition~\ref{prop:cellcharacterization}, the minimum $\min_{J \in \cJ} p(x, J)$ is achieved uniquely on $J$. Hence (2) is not true, proving one direction of Theorem~\ref{thm:n4-charac}.

    For the other direction, suppose that one of (1) or (2) is not true. If (1) is not true, then some facet of $[0,1]^4$ contains an edge of length 3, so $p$ is not a valuated $\Delta$-matroid. Suppose (2) is not true. Let $J \in \cJ$ be the argument at which the minimum is achieved. If $J \in \cJ_1$, then by Proposition~\ref{prop:cellcharacterization} $\cS_p$ contains an edge of length 4, so $p$ is not a valuated $\Delta$-matroid. If $J \in \cJ_2$ or $\cJ_3$, then $\cS_p$ contains an image of $[e_\emptyset, e_{123}]$ under $B_n$, and hence an edge of length 3. So $p$ is not a valuated $\Delta$-matroid, completing the proof.
\end{proof}

\section{A finite characterization for all $n$}\label{sec:characterization-for-all-n}

Here we show that in order to check if $(p_S)_{S\subseteq [n]}$ is a valuated $\Delta$-matroid, it is sufficient to only consider the $4$-dimensional faces of $[0,1]^n$. This is analogous to ordinary valuated matroids, where only the 3-dimensional faces of the hypersimplex need to be considered.

\begin{proposition}\label{prop:local-edge-length}
    Let $\cS$ be a subdivision of $[0,1]^n$ with vertices in $\{0,1\}^n$. Let $k$ be a positive integer. If $\cS$ has no edges of length $\ell$ for all $k < \ell \leq 2k$, then it has no edges of length $> 2k$.
\end{proposition}

\begin{proof}
Suppose for the sake of contradiction that $\cS$ has no edges of length $k < \ell \leq 2k$ and an edge of length $m > 2k$. Choose the smallest such edge $e$. By applying symmetry, we may assume $e$ has endpoints $e_{\emptyset}$ and $e_{[m]}$. Now let $\cS'$ be the restriction of $\cS$ to the face $[0,1]^m$ of $[0,1]^n$. Consider a $2$-dimensional cell $F$ of $\cS'$ which contains $e$. The polygon $F$ contains an edge between $e_{\emptyset}$ and some point of $\{0,1\}^m$ other than $e_{[m]}$, say $e_S$. By assumption $e$ is the shortest edge of length $> 2k$ and $\cS$ has no edges of length $k < \ell \leq 2k$, so $|S|\leq k$. Since $\dim F = 2$, we have that $F$ lies in the linear span of $e_S$ and $e_{[m]}$. Hence, the only other possible vertex of $F$ is $e_{[m]\setminus S}$.  Since $[0,e_{[m]}]$ is an edge, $F$ cannot contain $e_{[m]\setminus S}$, since $[e_{\emptyset},e_{[m]}]$ and $[e_S,e_{[m]\setminus S}]$ have the same midpoint.
Thus $F$ only has three vertices, so it contains the edge $[e_S,e_{[m]}]$. This edge has length $m-|S| > 2k-|S| \ge k$, and also length strictly less than $m$, a contradiction.  
\end{proof}

Applying \cref{prop:local-edge-length} with $k=2$ gives the following immediate corollary.
\begin{corollary}
    If a subdivision of $[0,1]^n$ with vertices in $\{0,1\}^n$ has no edges of length $3$ or $4$, then it has no edges of length $> 4$.
\end{corollary}

This gives us a characterization of valuated $\Delta$-matroids whose effective domain is the full cube.

\begin{corollary}
A function $p : \{0,1\}^n \to \R$ is a valuated $\Delta$-matroid if and only if its restriction to every 4-dimensional face of $[0,1]^n$ is a valuated $\Delta$-matroid.
\end{corollary}

For general valuated $\Delta$-matroids, we have the following characterization.

\begin{theorem}
A function $p :\{0,1\}^n \to \R\cup \{\infty\}$ is a valuated $\Delta$-matroid if and only if $\dom p$ is a $\Delta$-matroid and the restriction of $p$ to every 4-dimensional face of $[0,1]^n$ is a valuated $\Delta$-matroid.
\end{theorem}
\begin{proof}
    One direction is immediate: if $p$ is a valuated $\Delta$-matroid, then $\conv \dom p$ has no edges of length $ > 2$ and hence is a $\Delta$-matroid polytope. Moreover, the restriction of $p$ to any 4-face of the cube is also a valuated $\Delta$-matroid.

    To show the reverse direction, let $p :\{0,1\}^n \to \R\cup \{\infty\}$ be a function such that $\dom p$ is a $\Delta$-matroid and the restriction of $p$ to every 4-dimensional face of $[0,1]^n$ is a valuated $\Delta$-matroid. Suppose $\cS_p$ has an edge of length greater than or equal to 5, and let $e$ be such an edge of minimum length $m$. We may assume e has endpoints $e_{\emptyset}$ and $e_{[m]}$. Let $\cS'$ be the restriction of $\cS$ to the face $[0,1]^m$ of $[0,1]^n$. Since $\dom p$ is a $\Delta$-matroid and $\varnothing$, $[m] \in \dom p$, by the symmetric exchange property there exists some $j\in[m]$ such that $\{1,j\} \in \dom p$. In particular, $\cS'$ is not the single edge $e$, and therefore $e$ is contained in a two-dimensional cell of $\cS'$. The same argument as in \cref{prop:local-edge-length} leads to a contradiction.
\end{proof}

\section{The rank function}
One way to characterize matroids is via rank functions. In fact, the rank function of any matroid gives a valuated matroid \cite[Proposition 4.4]{speyer2004tropical}, so in particular, any matroid polytope appears as a cell in some matroidal subdivision.

Let $M = ([n],\mathcal{F})$ be a $\Delta$-matroid, where $\mathcal{F}$ is the set of bases.
Following \cite[Section 2.2]{larson2023rank}, we define the rank function for a $M$ as $r_M : 2^{[n]} \to \mathbb{Z}_{\geq 0}$ where
\[
    r_M(S) = \max_{B\in\mathcal{F}} ( |B\cap S| + |B^C \cap S^C| ).
\]
By construction, $r_M(S) \leq n$ for all $S\subseteq[n]$, with equality if and only if $S\in\mathcal{F}$.

In this section, we show that if $M$ is a $\Delta$-matroid, then $-r_M$ is a valuated $\Delta$-matroid. Thus, just as in the matroid case, any $\Delta$-matroid polytope appears as a cell in some $\Delta$-matroidal subdivision.

First, we give a geometric interpretation of $r_M$.
\begin{proposition}
    The rank function $r_M(S)$ is equal to
    \[
        r_M(S) = n - \min_{B\in\mathcal{F}} |e_B - e_S|_1.
    \]
     That is, $r_M(S)$ also has a geometric interpretation in terms of how far $e_S$ is from the closest vertex of $P_M$.
\end{proposition}
\begin{proof}
    We note that $|e_B - e_S|_1 = |B\setminus S| + |S\setminus B|$. Since $|B\cap S| + |B^C \cap S^C| + |B\setminus S| + |S\setminus B| = n$, then
    \begin{align*}
        r_M(S) &= \max_{B\in\mathcal{F}} ( |B\cap S| + |B^C \cap S^C| )\\
            &= \max_{B\in\mathcal{F}}\left(n - \left(|B\setminus S| + |S\setminus B|\right)\right)\\
            &= n - \min_{B\in\mathcal{F}}(|B\setminus S| + |S\setminus B|)\\
            &= n - \min_{B\in\mathcal{F}} |e_B-e_S|_1.
    \end{align*}
\end{proof}

In particular, this means that $|r_M(S) - r_M(S\Delta i)| \leq 1$ by the triangle inequality, so adding or removing a single element from $S$ changes the rank by at most one.

\begin{lemma}\label{lem:if-rank-max-unique-then-feasible}
    Suppose $n \geq 3$, and let $M=([n], \mathcal{F})$ be a $\Delta$-matroid. If $\max_{S\subseteq[n]}(r_M(S) + r_M(S^C))$ is attained uniquely by some $S \subseteq [n]$, then $S,S^C \in \mathcal{F}$.
\end{lemma}
\begin{proof}
    Assume $r_M(S) + r_M(S^C)$ attains its maximum at $S$ and suppose $S\notin\mathcal{F}$. We wish to show that $\max_{S\subseteq[n]}(r_M(S) + r_M(S^C))$ is not attained uniquely.
    
    Let $B\in\mathcal{F}$ be a set in $\mathcal{F}$ closest to $S$ (i.e.\ $B$ attains $\min_{B\in\mathcal{F}}|e_B-e_S|_1$), and let $i \in S\Delta B$. Then $|e_B - e_{S\Delta i}| = |e_B - e_S| - 1$, so $r(S\Delta i) = r(S) + 1$. Moreover, adding or removing an element can drop the rank by at most one, so $r(S^C \Delta i) \geq r(S^C) - 1$. Since $S^C\Delta i = (S\Delta i)^C$,
    \[
        r(S\Delta i) + r((S\Delta i)^C) \geq r(S) + r(S^C).
    \]
    We know that $S\Delta i \neq S$, and since $n \geq 3$, we have $S\Delta i \neq S^C$, so $r_M(S)+r_M(S^C)$ cannot attain $\max_{S\subseteq [n]}(r_M(S) + r_M(S^C))$ uniquely.
\end{proof}

Now, we are ready to show that $-r_M$ is a valuated $\Delta$-matroid.

\begin{theorem}
    Let $M = ([n], \mathcal{F})$ be a $\Delta$-matroid and let $r_M:2^{[n]} \to \Z_{\geq 0}$ be its rank function. Then $-r_M$ is a valuated $\Delta$-matroid.
\end{theorem}
\begin{proof}
    If $n = 2$, all functions on $2^{[2]}$ are valuated $\Delta$-matroids, so it suffices to consider the case where $n \geq 3$.

    Let $M = ([n],\mathcal{F})$ be a $\Delta$-matroid and suppose that the subdivision induced by $-r_M$ has an edge of length greater than 2. By replacing $M$ with its restriction to the face of $[0,1]^n$ containing this edge, we may assume that this subdivision contains an edge of length $n$.
    It follows that $\max_{S\subseteq[n]}(r_M(S) + r_M(S^C))$ is attained uniquely. By \cref{lem:if-rank-max-unique-then-feasible}, this implies that it is attained uniquely by some $T,T^C\in\mathcal{F}$.
    
    Since $r_M(S) \leq n$ for all $S$ with equality if and only if $S\in\mathcal{F}$, the $\Delta$-matroid polytope $P_M = \conv\{e_S: S\in\mathcal{F}\}$ must appear as a cell in the induced subdivision, and this cell contains both $e_T$ and $e_{T^C}$. Since $M$ is a $\Delta$-matroid, all edges of $P_M$ have length at most two, so $[e_T,e_{T^C}]$ cannot be an edge of the subdivision, a contradiction.
    Thus, $-r_M$ cannot induce an edge of length $n$, completing the proof.
\end{proof}

\section{Dimension}\label{sec:dim}
We define $\Delta$-$\Dr(n)$ to be the set of all functions $p : \{0,1\}^n \to \R$ which are valuated $\Delta$-matroids. We call this the \emph{$\Delta$-Dressian} on $[n]$. $\Delta$-$\Dr(n)$ is a polyhedral conical complex in the vector space of all functions $p : \{0,1\}^n \to \R$.

In this section we consider the dimension of $\Delta$-$\Dr(n)$; that is, the largest $d$ such that $\Delta$-$\Dr(n)$ contains a subset affinely equivalent to a $d$-dimensional ball. We note that for $n \ge 3$, $\Delta$-$\Dr(n)$ is not a pure polyhedral complex; that is, the maximal cones are not all the same dimension. For example, consider the two subdivisions in \cref{fig:maximal-subdivisions-3}. The functions which induce these subdivisions form (the relative interiors of) two maximal cones in $\Delta$-$\Dr(3)$. However, the cone for the subdivision on the left has dimension seven, while the subdivision on the right is a triangulation, so its cone is full- (eight-) dimensional.

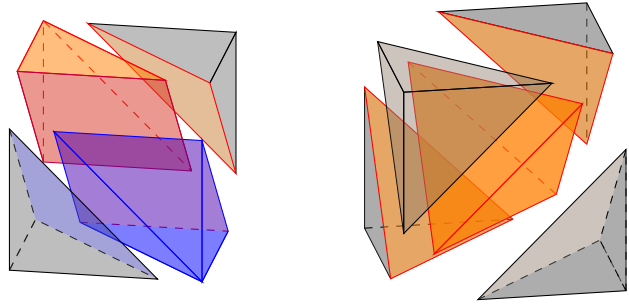
\begin{figure}[h]
    \centering
    \begin{minipage}{.3\textwidth}
        \centering
        \tdplotsetmaincoords{70}{100}
\begin{tikzpicture}[cube/.style={fill=gray, fill opacity=0.3}, simplex/.style={color=red,fill=orange,fill opacity=0.3}, tdplot_main_coords]
    \def\R{2}
    \def\r{0.3}

    \begin{scope}[shift={(-\r,\r,\r)}]
        \fill[cube] (0,0,\R) -- (0,\R,\R) -- (0,\R,0);
        \draw[cube] (0,0,\R) -- (0,\R,\R)-- (\R,\R,\R);
        \draw[cube] (\R,\R,\R) -- (0,\R,\R) -- (0,\R,0);
        \draw[simplex] (0,0,\R)--(0,\R,0)--(\R,\R,\R)--cycle;
    \end{scope}

    \begin{scope}[shift={(-\r,-\r,\r)}]
        \draw[simplex,dashed] (0,0,0) -- (0,0,\R) -- (0,\R,0);
        \draw[simplex] (0,0,\R)--(\R,\R,\R)--(0,\R,0);
        \draw[simplex] (0,0,\R) -- (\R,0,\R) -- (0,0,0);
        \draw[simplex] (0,0,\R) -- (\R,\R,\R) -- (\R,0,\R) -- cycle;
        \draw[simplex, color=purple] (0,0,0) -- (\R,0,\R) -- (\R,\R,\R) -- (0,\R,0)--cycle;
    \end{scope}

    \begin{scope}[shift={(\r, \r, -\r)}, simplex/.append style={color=blue}]
        \fill[simplex, color=purple] (0,0,0) -- (\R,0,\R) -- (\R,\R,\R) -- (0,\R,0)--cycle;
        \draw[simplex, color=purple, dashed] (0,0,0) -- (0,\R,0);

        \draw[simplex] (0,0,0) -- (\R,\R,0) -- (0,\R,0);
        \draw[simplex] (0,0,0) -- (\R,0,\R) -- (\R, \R, 0) -- cycle;
        \draw[simplex] (\R,\R,\R) -- (\R,\R,0) -- (0,\R,0)--cycle;
        \draw[simplex] (\R,\R,\R) -- (\R,0,\R) -- (\R,\R,0)--cycle;
    \end{scope}

    \begin{scope}[shift={(\r,-\r,-\r)}]
        \fill[simplex, color=blue] (\R,0,\R) -- (0,0,0) -- (\R,\R,0)--cycle;
        \draw[cube, dashed](\R,0,\R) -- (0,0,0) --(\R,0,0);
        \draw[cube, dashed] (\R,0,0) -- (0,0,0) -- (\R,\R,0);
        \draw[cube] (\R,0,0) -- (\R,0,\R) -- (\R,\R,0)--cycle;
    \end{scope}
    
\end{tikzpicture}
    \end{minipage}%
    \begin{minipage}{.3\textwidth}
        \centering
  
\tdplotsetmaincoords{70}{80}
\begin{tikzpicture}[cube/.style={fill=gray, fill opacity=0.4}, simplex/.style={color=red, fill=orange, fill opacity=0.5}, tdplot_main_coords]
    \def\R{2}
    \def\r{0.5}

    \begin{scope}[shift={(-\r, -\r, -\r)}]
        \draw[cube] (\R,0,0) --(0,0,0) -- (0,0,\R);
        \fill[cube] (0,\R,0) -- (0,0,0) -- (0,0,\R);
        \fill[cube] (0,\R,0) -- (0,0,0) -- (\R,0,0);
        \draw[cube,dashed] (0,0,0) -- (0,\R,0);
        \draw[simplex] (0,0,\R) -- (0,\R,0) -- (\R,0,0)--cycle;
    \end{scope}

    \begin{scope}[shift={(-\r,\r,\r)}]
        \fill[cube] (0,\R,\R) -- (0,0,\R) -- (0, \R, 0) -- cycle;
        \fill[cube] (0,\R,\R) -- (\R,\R,\R) -- (0, \R, 0) -- cycle;
        \draw[cube, dashed] (0,\R,\R) -- (0,\R,0);
        \draw[cube] (0,\R,\R) -- (\R,\R,\R) -- (0, 0, \R) -- cycle;
        \draw[simplex] (0,0,\R) -- (\R,\R,\R)-- (0,\R,0) -- cycle;
    \end{scope}

    \begin{scope}[shift={(0,0,0)}]
        \fill[simplex] (\R,0,0) -- (0,0,\R) -- (0,\R,0)--cycle;
        \draw[simplex, dashed] (0,0,\R) -- (0,\R,0);

        \fill[simplex] (0,0,\R) -- (\R,\R,\R) -- (0,\R,0)--cycle;

        \draw[simplex](0,0,\R)--(\R,0,0)--(\R,\R,\R)--cycle;
        \draw[simplex](\R,\R,\R) -- (\R,0,0) -- (0,\R,0)--cycle;
    \end{scope}

    \begin{scope}[shift={(\r,-\r,\r)}]
        \fill[simplex, opacity=0.1](0,0,\R)--(\R,0,0)--(\R,\R,\R)--cycle;
        \draw[cube](0,0,\R)--(\R,0,0);

        \draw[cube](0,0,\R)--(\R,0,\R)--(\R,0,0);
        \draw[cube](0,0,\R)--(\R,0,\R)--(\R,\R,\R)--cycle;
        \draw[cube](\R,0,0)--(\R,0,\R)--(\R,\R,\R)--cycle;
    \end{scope}

    \begin{scope}[shift={(\r,\r,-\r)}]
        \fill[simplex, opacity=0.1](\R,\R,\R) -- (\R,0,0) -- (0,\R,0)--cycle;

        \draw[cube, dashed] (\R,0,0) -- (0,\R,0) -- (\R,\R,0);
        \draw[cube] (\R,0,0) -- (\R,\R,0) -- (\R,\R,\R)--cycle;
        \draw[cube, dashed] (\R,\R,0) -- (0,\R,0) -- (\R,\R,\R)--cycle;
    \end{scope}
\end{tikzpicture}
    \end{minipage}
    \caption{Two $\Delta$-matroidal subdivisions of the cube, each corresponding to a maximal cone in $\Delta$-$\Dr(3)$.}
    \label{fig:maximal-subdivisions-3}
\end{figure}

By the above discussion, the dimension of $\Delta$-$\Dr(3)$ is eight, i.e. $\Delta$-$\Dr(3)$ is full-dimensional. Now, we show that the codimension of $\Delta$-$\Dr(4)$ is 1, and that in general, the dimension of $\Delta$-$\Dr(n)$ is exponential in $n$.

\begin{proposition}\label{prop:DimDr4}
    $\Delta$-$\Dr(4)$ has codimension $1$ in $\R^{2^{4}}\cong \R^{16}$.
\end{proposition}
\begin{proof}
    Consider the subdivision induced by $p_\emptyset = p_{12} = p_{34} = p_{1234} = 0$, $p_i = p_{ijk} = 100$ for all distinct $i,j,k\in [4]$, and $p_{ij} = 1$ for all $ij \neq 12, 34$. We claim that this induces a $\Delta$-matroid subdivision, and that the associated cone in the $\Delta$-Dressian has codimension $1$.

    First, we claim that this induces a $\Delta$-matroid subdivision. Indeed, since $\min_{S\subseteq [4]} p_S + p_{S^C}$ is obtained by $[e_\emptyset, e_{1234}]$ and $[e_{12}, e_{34}]$, there are no edges of length $4$.
    Moreover, on any $3$-dimensional face of the form $x_l = 0$, setting $\{i,j,k\} = [4] \setminus l$, we have
    \begin{align*}
        2(p_\emptyset + p_{ijk}) &= 200\\
        2(p_m + p_{\{i,j,k\} \setminus m}) &\ge 200 \text{ for all } m \in \{i,j,k\} \\
        p_\emptyset + p_{ij} + p_{ik} + p_{jk} &\le 3 \\
        p_i + p_j + p_k + p_{ijk} &= 400,
    \end{align*}
    so the minimum is attained uniquely by $p_\emptyset + p_{ij} + p_{ik} + p_{jk}$. Similarly, on the $3$-dimensional face $x_l = 1$, setting $\{i,j,k\} = [4] \setminus l$, we have
    \begin{align*}
        2(p_l + p_{1234}) &= 200\\
        2(p_{lm} + p_{[4] \setminus m}) &\ge 200 \text{ for all } m \in \{i,j,k\} \\
        p_l + p_{ijl} + p_{jkl} + p_{ikl} &= 400\\
        p_{il} + p_{jl} + p_{kl} + p_{1234} &\le 3,
    \end{align*}
    so the minimum is attained uniquely by $p_{il} + p_{jl} + p_{kl} + p_{1234}$. Thus, by Lemma~\ref{lem:valuated-delta-matroid-3}, there are no edges of length $3$, so $p$ is a valuated $\Delta$-matroid.

    Moreover, we claim that the functions which induce $\cS_p$ form a cone of codimension $1$ in $\Delta$-$\Dr(4)$. Let $x = (1/2,1/2,1/2,1/2)$. First, we claim that $\min_{J\in\mathcal{J}} p(x,J)$ is attained exactly twice. Indeed, since $p_S > 0$ for all $S \neq \varnothing, 12, 34, 1234$, then $p(x,J) > 0$ for all $J \in \mathcal{J}$ except for $[e_\emptyset, e_{1234}]$ and $[e_{12}, e_{34}]$, and so the minimum is attained exactly twice, as desired.

    Moreover, we showed above that the minimum from Lemma~\ref{lem:valuated-delta-matroid-3} on each three dimensional face is attained uniquely. Hence, the only linear relation satisfied by functions which induce $\cS_p$ is $p_\emptyset + p_{1234} = p_{12} + p_{34}$. Thus, the cone of such functions has codimension 1.

    Finally, $\Delta$-$\Dr(4)$ is not full-dimensional, because any $\Delta$-matroid must satisfy (2) from Theorem~\ref{thm:n4-charac}.
\end{proof}

In fact, the dimension of the $\Delta$-Dressian is exponential in $n$.
\begin{theorem}\label{thm:DimDrBounds}
    The dimension of the $\Delta$-Dressian on $[n]$ is between $\frac{1}{2}2^n$ and $\frac{15}{16}2^n$.
\end{theorem}
\begin{proof}
    First, we claim that the $\Delta$-Dressian has codimension at least $2^{n-4}$, which gives the upper bound of $\frac{15}{16}2^n$ on the dimension. Indeed, we note that we can find $2^{n-4}$ disjoint 4-faces of $[0,1]^n$. For each $b_5,\dots,b_n\in\{0,1\}$, we consider the 4-dimensional face $\{x\in[0,1]^n : x_5 = b_5, \dots, x_n = b_n\}$. Each of these faces contributes a codimension-1 constraint on the $\Delta$-Dressian, and these constraints are independent since the faces are pairwise disjoint. Thus, the dimension of the $\Delta$-Dressian is at most $2^n - 2^{n-4} = \frac{15}{16}2^n$.

    Next, we claim the dimension is at least $\frac{1}{2}2^n$. Consider the subdivision induced by taking $p_S = 0$ whenever $|S|$ is even, and $p_S > 0$ otherwise. We claim that this is a $\Delta$-matroid subdivision. Indeed, on any 3-face or 4-face of the cube, the minimum of $p$ is attained uniquely on the convex hull of all vertices with an even sum of coordinates, which is a polytope of dimension 3 or 4 respectively. Hence the induced subdivision has no edges of length 3 or 4, so it $\Delta$-matroid subdivision as claimed. Since we allow any positive value on the odd vertices, the cone corresponding to this subdivision has dimension at least $2^{n-1}$, as desired.
\end{proof}

\section{Representability via Hermitian matrices}

Let $K$ be a field with a nonarchimedean valuation $\nu:K\to \R\cup \{\infty\}.$ The image $\Gamma$ of $K^*$ under valuation is an additive subgroup of $\R$ called the \emph{value group}.
We use $\kk$ to denote the \emph{residue field} of $K$, 
which is the quotient of the valuation ring $\mathcal{O} = \{a\in K : \nu(a)\geq 0\}$ by its unique maximal ideal $\mathfrak{m} = \{a\in K : \nu(a)> 0\}$. Given an element $a\in \mathcal{O}$, we use $[a]_{\mathfrak{m}}$ to denote its image in $\kk = \mathcal{O}/\mathfrak{m}$. 

Throughout this section, we will assume that the field $K$ has  an automorphic involution $a\mapsto \overline{a}$ satisfying $\nu(a) = \nu(\overline{a})$ with fixed field $F$.
We include the possibility of this involution being the identity map and $F = K$. 
Since the valuation is invariant under the involution $a\mapsto \overline{a}$, the valuation ring $\mathcal{O}$ and  maximal ideal $\mathfrak{m}$ are invariant under this involution. The involution on $K$ therefore induces an automorphic involution of the residue field $\kk$, which we also denote by $a\mapsto \overline{a}$. 

We furthermore assume that the value group $\Gamma$ is dense in $\R$. Since $\nu(a\overline{a})=2\nu(a)$, this implies that the image of $F^*$ under the valuation is also dense in $\R$. Without loss of generality we can do this as follows. 
If the original valuation $\nu$ on $K$ is trivial, then we can replace $K$ with the field $K\{\!\{t\}\!\}$ of Puiseux series over $K$, with the automorphic involution $\overline{\sum_q a_qt^q}=\sum_q \overline{a_q}t^q$ and the valuation $\nu(\sum_q a_qt^q) = \min\{q : a_q\neq 0\}$. If $\Gamma$ is non-zero but discrete then  $\Gamma=\alpha\Z$ for some $\alpha\in \R_{>0}$. Choose $\beta\in \R\backslash (\alpha\Q)$. By \cite[Theorem 2.2.1.]{ValuedFields} we can extend the valuation on $K$ to a transcendental extension $K(t)$ by taking $\nu(\sum_i a_i t^i) = \min_i\{\nu(a_i)+i\beta\}$ and $\nu(f/g)=\nu(f)-\nu(g)$. The automorphic involution also extends to $K(t)$ acting on coefficients and fixing $t$. 
The image of the valuation is then $\alpha\Z+\beta \Z$, which is dense in $\R$.

The main result of the section is that, with the expected definition of Hermitian, the valuations of principal minors of Hermitian and skew-Hermitian matrices give a valuated $\Delta$-matroid. 
Specifically, we call a matrix $A\in K^{n\times n}$ \emph{Hermitian} if $\overline{A} = A^T$. Similarly, $A$ is \emph{skew-Hermitian} if $\overline{A} = -A^T$. 
For an $n\times n$ matrix $A$ and subset $S\subseteq [n]$, we use $A_S$ to denote the determinant of the $|S|\times |S|$ principal submatrix of $A$
whose rows and columns are indexed by $S$.

\begin{theorem}\label{thm:HermValDelta}
    If $A$ is an $n \times n$ Hermitian or skew-Hermitian matrix over $K$, then the function $p : \{0,1\}^n \to \R \cup \{\infty\}$ given by $p_S = \nu(A_S)$ is a valuated $\Delta$-matroid. 
\end{theorem}

We need to build up to the proof of this theorem. 
The main idea resembles the proof of \cite[Corollary 4.3]{B07} that for any Hermitian or skew-Hermitian matrix over $\C$, the subsets corresponding to non-zero principal minors form a $\Delta$-matroid. 

We work with a generating polynomial for the principal minors, 
\begin{equation}\label{eq:fA}
f_A = \det({\rm diag}(x_1, \hdots, x_n) + A) = \sum_{S\subseteq [n]}A_Sx^{[n]\backslash S} \in K[x_1,\dots,x_n].
\end{equation}
In \cite{AV21, AV24}, the principal minor map was studied 
via \emph{Rayleigh differences}.
Formally, for $i,j\in [n]$ and $f \in K[x_1, \hdots, x_n]$ define
\[
\Delta_{ij}(f) = \frac{\partial f}{\partial x_i}\cdot \frac{\partial f}{\partial x_j} - f \cdot \frac{\partial^2 f}{\partial x_i \partial x_j }.
\]
If we write $f = x_ix_j f_{ij} + x_if_i^j + x_jf_j^i + f^{ij}$ for some polynomials $f_{ij}$, $f_i^j$, $f_j^i$, $f^{ij}$ that do not involve $x_i$ and $x_j$, this expression simplifies to $f_{i}^jf_j^i - f_{ij}f^{ij}$. We say a polynomial is \emph{multiaffine} if it has degree $\leq 1$ in each variable.

In fact, we prove the following slightly stronger version of \Cref{thm:HermValDelta}:

\begin{theorem}\label{thm:HermValDelta2}
    Let $A$ be an $r \times r$ Hermitian or skew-Hermitian matrix over $K$ and let $v_1, \hdots, v_n$ be vectors $K^r$. Define $f\in F[x_1, \hdots, x_n]$ to be the determinantal polynomial
    \[ f = \sum_{S\subseteq [n]}c_Sx^{[n]\backslash S} = \det\left(\sum_{i=1}^n x_i v_i \overline{v_i}^T + A\right).\]
    Then the function $p : \{0,1\}^n \to \R \cup \{\infty\}$ given by $p_S = \nu(c_S)$ is a valuated $\Delta$-matroid. 
\end{theorem}
Note that \Cref{thm:HermValDelta} follows by taking $r=n$ and $v_i = e_i$ for $i=1, \hdots, n$.

\begin{lemma} \label{lem:factorization}
Let $A\in K^{r\times r}$ be with $\overline{A} = \sigma A^T$ with $\sigma\in \{\pm1\}$ and $v_1, \hdots, v_n \in K^r$. Define
$$ f = \sum_{S\subseteq [n]}c_Sx^{[n]\backslash S} = \det\left(\sum_{i=1}^n x_i v_i \overline{v_i}^T + A\right).$$ 
For any $i,j\in [n]$, 
$\Delta_{ij}(f)$ factors as $\sigma^{n-1}g_{ij}(x)\overline{g_{ij}}(\sigma x)$ for some multiaffine polynomial $g_{ij}\in K[x_1,\hdots, x_n]$. 
\end{lemma}
\begin{proof}
Without loss of generality, we take $i=1$ and $j=2$. 
First suppose that $v_1$ and $v_2$ are linearly independent in $K^r$. 
Then they can be extended to a basis of $K^r$ that we collect into the columns of a matrix $U$. The determinant of linear matrices $M$ and $U^{-1}
M(\overline{U}^{-1})^T$ differ by only a constant scalar, so we may assume without loss of generality that $v_1=e_1$ and $v_2 = e_2$. 
Dodgson condensation states that for any $r\times r$ matrix $M$, 
\[\det(M(1,1))\det(M(2,2)) - \det(M)\det(M(12, 12)) = \det(M(1,2))\det(M(2,1))\]
where $M(S,T)$ is obtained from $M$ by \emph{dropping} rows $S$ and columns $T$. Consider the linear matrix $M = \sum_{k=1}^n x_k v_k \overline{v_k}^T + A$  and polynomial $f = \det(M)$. 
For $S\subseteq \{1,2\}$, $\det(M(S,S))$ is precisely $\left(\prod_{k\in S}\frac{\partial}{\partial x_k}\right)f$. Therefore $\Delta_{12}(f)$ factors as $g_{12}\cdot g_{21}$ where $g_{12} = \det(M(1,2))$ and $g_{21} = \det(M(2,1))$. 

If $\overline{A} = \sigma A^T$, then 
\[M^T  = \sum_{k=1}^n x_k \overline{v_k}v_k^T+\sigma\overline{A}
 = \sigma\left(\sum_{k=1}^n \sigma x_k \overline{v_k}v_k^T+\overline{A}\right) = \sigma \overline{M}(\sigma x). \label{eq:transpose} \]
It follows that 
$$g_{21} = \det(M(2,1)) = \det(M^T(1,2)^T) = \sigma^{n-1}\overline{g_{12}}(\sigma x).$$

Finally, if $v_1$ and $v_2$ are linearly dependent, then, up to relabeling,  $v_1 = \lambda v_2$ for some $\lambda \in K$. Following the proof of \cite[Theorem 6.1]{AV21}, $v_2\overline{v_2}^T = \lambda \overline{\lambda} v_1\overline{v_1}^T$, showing that $f(x_1, \hdots, x_n) = f(0,\lambda \overline{\lambda}x_1+x_2 ,x_3, \hdots, x_n)$.  
Taking partial derivatives shows that $\frac{\partial f}{\partial x_1} = \lambda \overline{\lambda}\frac{\partial f}{\partial x_2}$ and that $\frac{\partial^2 f}{\partial x_1\partial x_2} = 0$. Then $\Delta_{12}(f) = \lambda \overline{\lambda} (\frac{\partial f}{\partial x_2})^2$. 
If $\overline{A} = \sigma A^T$, we note that since $\frac{\partial f}{\partial x_2}$ has a determinantal representation of the same form as $f$ but of size $n-1$,  we have $\frac{\partial f}{\partial x_2}= \sigma^{n-1}\overline{\frac{\partial f}{\partial x_2}}(\sigma x)$ by the computation in \eqref{eq:transpose}. Thus $\Delta_{12}(f) = \sigma^{n-1}g_{ij}(x)\overline{g_{ij}}(\sigma x)$ with $g_{12} = \lambda  \frac{\partial f}{\partial x_2}$.
\end{proof}

We will show that such factorizations pass to factorizations in $\kk[x_1, \hdots, x_n]$. 
When $K$ has a splitting this can be interpreted as factorizations passing to initial forms, but we avoid the assumption that $K$ has a splitting. 

For $\lambda=(\lambda_1, \hdots, \lambda_n)\in (K^*)^n$ and $g = \sum c_\alpha x^\alpha \in K[x_1, \hdots, x_n]$, define 
\[\lambda \star g = g(\lambda_1x_1, \hdots, \lambda_n x_n) = \sum c_\alpha \lambda^{\alpha}x^\alpha.\]

\begin{lemma}\label{lem:InitialRayleigh}
Let $f \in K[x_1, \hdots, x_n]$, $\lambda\in (K^*)^n$ and $i,j\in [n]$. Then 
\[ \Delta_{ij}(\lambda \star f) = \lambda_i\lambda_j (\lambda \star \Delta_{ij}(f)).\]
\end{lemma}
\begin{proof}
This follows from the chain rule, since $\frac{\partial}{\partial x_i}(\lambda \star f) = \lambda_i (\lambda \star \frac{\partial}{\partial x_i}f)$. 
\end{proof}

\begin{lemma}\label{lem:InitialRayleigh2}
Let $f  \in \mathcal{O}[x_1, \hdots, x_n]$ be multiaffine and $i,j\in [n]$. Then
\[[\Delta_{ij}(f)]_{\mathfrak{m}} = \Delta_{ij}([f]_{\mathfrak{m}}).\]
Suppose further that $[\Delta_{ij}(f)]_{\mathfrak{m}}$ is nonzero and that  $\Delta_{ij}(f)$ factors as $g\cdot h$ in $K[x_1, \hdots, x_n]$ where $g = \sum_S a_S x^S$, $h=\sum_T b_Tx^T$ are multiaffine with ${\rm argmin}_S\nu(a_S) = {\rm argmin}_T\nu(b_T)$.  Then $\Delta_{ij}([f]_{\mathfrak{m}})$ has factorization $\tilde{g}\cdot \tilde{h}$ in $\kk[x_1, \hdots, x_n]$ where $\tilde{g}$ and $\tilde{h}$ are multiaffine polynomials with $\supp(\tilde{g}) = \supp(\tilde{h})$.
\end{lemma}
\begin{proof}
As above, we write $f = x_ix_j f_{ij} + x_if_i^j + x_jf_j^i + f^{ij}$ for some polynomials $f_{ij}$, $f_i^j$, $f_j^i$, $f^{ij}$  in $K[x_1, \hdots, x_n]$ 
that do not involve $x_i$ and $x_j$ and $\Delta_{ij}(f) = f_{i}^jf_j^i - f_{ij}f^{ij}$.

Since all of the coefficients of $f$ belong to $\mathcal{O}$, the polynomials $f_{ij}$, $f_i^j$, $f_j^i$, $f^{ij}$,  and $\Delta_{ij}(f)$ all belong to $\mathcal{O}[x_1, \hdots, x_n]$. Moreover, $[f]_{\mathfrak{m}} = x_ix_j [f_{ij}]_{\mathfrak{m}} + x_i[f_i^j]_{\mathfrak{m}} + x_j[f_j^i]_{\mathfrak{m}} + [f^{ij}]_{\mathfrak{m}}$. Since the map taking $h\in \mathcal{O}[x_1, \hdots, x_n]$ to $[h]_{\mathfrak{m}}\in \kk[x_1, \hdots, x_n]$ is a ring homomorphism, we see that 
\[\Delta_{ij}([f]_{\mathfrak{m}}) = \left[f_i^j\right]_{\mathfrak{m}}\left[f_j^i\right]_{\mathfrak{m}} - \left[f_{ij}\right]_{\mathfrak{m}}\left[f^{ij}\right]_{\mathfrak{m}} =\left[f_i^jf_j^i - f_{ij}f^{ij}\right]_{\mathfrak{m}} = [\Delta_{ij}(f)]_{\mathfrak{m}}.\]

For the second claim suppose $[\Delta_{ij}(f)]_{\mathfrak{m}}$ is nonzero and that $\Delta_{ij}(f)$ factors as $g\cdot h$ where $g = \sum_S a_S x^S$, $h=\sum_T b_Tx^T$ are multiaffine with ${\rm argmin}_S\nu(a_S) = {\rm argmin}_T\nu(b_T)$ and assume that $[\Delta_{ij}(f)]_{\mathfrak{m}}$ is nonzero. First, we claim that, after rescaling, we can take both $g$ and $h$ in $\mathcal{O}[x_1, \hdots, x_n]$. 

Indeed, suppose that $U\in  {\rm argmin}_{S}\nu(a_S) ={\rm argmin}_{T} \nu(b_T)$.
Then both $a_{U}^{-1}g(x)$ and $b_{U}^{-1}h(x)$ belong to $\mathcal{O}[x_1, \hdots, x_n]$ and have nonzero image mod $\mathfrak{m}$. 
Therefore so does their product 
$a_{U}^{-1}b_{U}^{-1}g(x)h(x) = a_{U}^{-1}b_{U}^{-1}\Delta_{ij}(f)$. 

Since $[\Delta_{ij}(f)]_{\mathfrak{m}}$ is nonzero, the minimum valuation of the coefficients of $\Delta_{ij}(f)$ is zero. The above argument shows that the minimum valuation of the  coefficients of $a_{U}^{-1}b_{U}^{-1}\Delta_{ij}(f)$ is also zero, implying that $\nu(a_U^{-1}b_U^{-1})=0$. 
Therefore $\nu(a_U)=\nu(b_U^{-1})$.

It follows that both $a_{U}^{-1}g$ and $a_U h$ belong to $\mathcal{O}[x_1, \hdots, x_n]$ and have nonzero image mod $\mathfrak{m}$. We can then take $\tilde{g}=[a_{U}^{-1}g]_{\mathfrak{m}}$ and $\tilde{h} = [a_{U} h]_{\mathfrak{m}} $.
The support of $\tilde{g}$ and $\tilde{h}$ will be ${\rm argmin}_S\nu(a_S)$  and ${\rm argmin}_T\nu(b_T)$, respectively, which are assumed to be equal. 
\end{proof}

\begin{proof}[Proof of \Cref{thm:HermValDelta2}]
Let $f = \sum_{S\subseteq [n]} c_Sx^{[n]\backslash S}\in K[x_1, \hdots, x_n]$ and  $p:\{0,1\}^n \to \R\cup\{\infty\}$ be as given in the statement of theorem. Let $[e_T,e_{T'}]$ be an edge of the regular subdivision of $[0,1]^n$ induced by the function $p$, and suppose that $(w,1)\in \R^{n+1}$ is the inner normal vector of a hyperplane supporting the corresponding lower face of $P_p$. That is, 
\begin{equation}\label{eq:normalVector}
\langle (w,1), (e_T, p_T)\rangle = \langle (w,1), (e_{T'}, p_{T'})\rangle < \langle (w,1), (e_S, p_S)\rangle 
\end{equation}
for all $S\neq T, T'$. 

We claim that we can choose $w\in \nu(F^*)^n$. 
Without loss of generality, we can assume that $1\in T \setminus T'$. Note that $w = (p_{T'}-p_{T})e_1$ satisfies the equation $\langle (w,1), (e_T, p_T)\rangle = \langle (w,1), (e_{T'}, p_{T'})\rangle$ and that any other solution differs by an vector $u$ with $\langle u, e_T-e_{T'}\rangle =0$. 
Since $e_T - e_{T'}\in \{0,\pm 1\}^n$, the solutions $u = (u_1, \hdots, u_n)$ of $\langle u, e_T-e_{T'}\rangle =0$ over $\nu(F^*)^n$ are dense in those over $\R^n$. Indeed, we can explicitly parametrize solutions by choosing arbitrary values for $u_2, \hdots, u_n$ and writing $u_1$ as the appropriate integer combination of these values. 
Since $(p_{T'}-p_{T})e_1 = \nu(c_{T'}/c_T, 1, \hdots, 1)$ also belongs to $\nu(F^*)^n$, we see that the solutions to 
$\langle (w,1), (e_T, p_T)\rangle = \langle (w,1), (e_{T'}, p_{T'})\rangle$ over $\nu(F^*)$ are dense in those over $\R$. In particular, we can find a solution $w\in \nu(F^*)^n$ to \eqref{eq:normalVector}.

It follows that there exists $\lambda\in (F^*)^n$ with $\nu(\lambda) = w$. 
Suppose, for the sake of contradiction, that $|T\Delta T'|\geq 3$. Without loss of generality, we can take distinct elements $i,j\in T'\setminus T$.  
Note that $\tilde{f} = (\lambda^{-T}c_T^{-1})(\lambda\star f)$ belongs to $\mathcal{O}[x_1, \hdots, x_n]$. Its image under 
the ring homomorphism $\mathcal{O}[x_1, \hdots, x_n] \to \kk[x_1, \hdots, x_n]$ induced by $\mathcal{O}\to\kk = \mathcal{O}/\mathfrak{m}$ has the form $[\tilde f]_{\mathfrak{m}} = a_Tx^T + a_{T'}x^{T'}$ where $a_T, a_{T'}\in \kk$ are nonzero.

By \Cref{lem:factorization}, the Rayleigh difference of $f$ factors as 
$\Delta_{ij}(f)=\sigma^{n-1}g_{ij}(x)\overline{g_{ij}}(\sigma x)$ for some multiaffine polynomial $g_{ij}$ and $\sigma\in \{\pm1\}$.  Then
\begin{align*}
    \Delta_{ij}(\tilde{f}) 
    & = \mu^2  \Delta_{ij}(\lambda\star f)\\
    & = \mu^2 \lambda_i \lambda_j ( \lambda \star \Delta_{ij}(f))\\
    & = \sigma^{n-1}\mu^2 \lambda_i \lambda_j (\lambda \star g_{ij}(x))(\lambda \star \overline{g_{ij}}(\sigma x))
\end{align*}
where $\mu = \lambda^{-T}c_T^{-1}$. Note that 
because $\nu$ is invariant under the involution $a\mapsto \overline{a}$, the set of terms in $\sigma^{n-1}\mu^2\lambda_i\lambda_j (\lambda \star g_{ij})$ and $(\lambda \star \overline{g_{ij}}(\sigma x))$ achieving the minimum valuation will be the same.

By assumption 
$[\tilde{f}]_{\mathfrak{m}}$ has the form $a_Tx^T + a_{T'}x^{T'}$ for some $a_T, a_{T'}\in \kk^*$. 
Then 
\begin{align*}
[\Delta_{ij}(\tilde{f})]_{\mathfrak{m}} & = \Delta_{ij}([\tilde{f}]_{\mathfrak{m}}) \\ & = \Delta_{ij}(a_Tx^T + a_{T'}x^{T'})\\ & = (a_Tx^T) \cdot (a_{T'}x^{T'\setminus \{i,j\}}) = a_Ta_{T'} x^{2(T\cap T')}x^{T\Delta T'\setminus \{i,j\}}
\end{align*}
Since $T\Delta T'\setminus \{i,j\}$ has at least one element, we see that this polynomial cannot factor as $ \tilde{g}\tilde{h}$ for multiaffine polynomials $\tilde{g}$, $\tilde{h}$ with the same support, contradicting \Cref{lem:InitialRayleigh2}. 
\end{proof}

\subsection{Examples}

\begin{example}
    One motivating example is the field $K = \C\{\!\{t\}\!\}$ of complex Puiseux series with complex conjugation, where $\overline{t} = t$. Here the fixed field is $F =  \R\{\!\{t\}\!\}$. The value group is $\Gamma=\Q$. 
Consider the $3\times 3$ Hermitian matrix  
\[
A = \begin{pmatrix}
     t & i & i \\
 -i & t & i\\
 -i & -i & t \\
\end{pmatrix}
\]
with entries in $K = \C\ps$.
The principal minors of $A$ are encoded by the coefficients of 
\[
f = \det({\rm diag}(x_1, x_2, x_3) +A) = 
x_1x_2x_3 + t(x_1x_2+x_1x_3+x_2x_3) +(-1 + t^2)(x_1+x_2+x_3) + (-3 t + t^3).
\]
In particular, $p_S= \nu(A_S)$ equals $0$ for $|S|\in \{0,2\}$, $1$ for $|S|\in \{1,3\}$. 
In this case $\Delta_{12}(f)$ factors as 
\[
\Delta_{12}(f) = x_3^2 +2tx_3 + 1+t^2
= (-1 + i t + i x_3)(-1 - i t - i x_3)
\]
Its image modulo $\langle t\rangle$ in $\C[x_1, x_2, x_3]$ factors as $(-1  + i x_3)(-1 - i x_3)$.
\end{example}

\begin{example}\label{ex:F2}
    Consider $K = \Q[\alpha]/\langle \alpha^2+\alpha+1 \rangle$ with the $2$-adic valuation on $\Q$ extended by $\nu(\alpha)=0$. 
    That is, $\nu(a+b\alpha) = \min\{\nu(a), \nu(b)\}$ where $a,b\in \Q$.
    This has an involution given by $\overline{\alpha} =-1-\alpha$ with fixed field $F = \Q$.  
    The value group is $\Gamma=\Z$.  As discussed in the introduction to this section, we could take a transcendental extension of $K$ if needed to make the value group dense in $\R$.
    Note that $\alpha \mapsto -1-\alpha$ gives a ring automorphism of the valuation ring $\mathcal{O}$ under which $\mathfrak{m}= \langle 2\rangle$ is invariant. The residue field $\kk$ is then $\mathbb{F}_2[\alpha]/\langle \alpha^2+\alpha+1 \rangle$ with induced involution $\alpha \mapsto 1+\alpha$.
    Consider the $3\times 3$ Hermitian matrix  
\[
A = \begin{pmatrix}
     1 & 1+2\alpha & 1+2\alpha \\
 -1-2\alpha & 1 & 1+2\alpha\\
 -1-2\alpha & -1-2\alpha & 1 \\
\end{pmatrix}
\]
with entries in $K = \Q[\alpha]/\langle \alpha^2+\alpha+1\rangle $. The principal minors of $A$ are encoded by the coefficients of 
\[
f = \det({\rm diag}(x_1, x_2, x_3) +A) = 
x_1x_2x_3 + x_1x_2+x_1x_3+x_2x_3 -2(x_1+x_2+x_3) - 8.
\]
In particular, $p_S= \nu(A_S)$ equals $0$ for $|S|\in \{0,1\}$, $1$ for $|S|=2$ and $3$ for $|S|=3$. 

In this case $\Delta_{12}(f)$ factors as 
\[
\Delta_{12}(f) = 3(x_3^2 + 2 x_3^2 + 4) = (
(1+2\alpha)x_3 -2 + 2 \alpha)((-1-2\alpha)x_3-4 - 2 \alpha)
\]
Its image modulo $2$ in $\kk[x_1, x_2, x_3]$ factors as $(x_3)^2$.

On the other hand, consider the $3\times 3$ skew-Hermitian matrix  
\[
A = \begin{pmatrix}
     0 & 1 & 1 \\
 -1 & 0 & 1+2\alpha\\
 -1 & 1+2\alpha & 0 \\
\end{pmatrix}
\]
with entries in $K = \Q[\alpha]/\langle \alpha^2+\alpha+1\rangle $.
Note that $\overline{1+2\alpha} = 1+2(-1-\alpha) = -(1+2\alpha)$. 
The principal minors of $A$ are encoded by the coefficients of 
\[
f = \det({\rm diag}(x_1, x_2, x_3) +A) = 
x_1x_2x_3 +3x_1+x_2+x_3+ 2 -4\alpha.
\]
In particular, $p_S= \nu(A_S)$ equals $0$ for $|S|\in \{0,2\}$, $\infty$ for $|S|=1$ and $1$ for $|S|=3$. 
\end{example}

\begin{example}\label{ex:F2_counterEx}
   To see that the assumption $\nu(a) = \nu(\overline{a})$ is necessary for  \Cref{thm:HermValDelta}, consider the field $K = \Q[\alpha]/\langle \alpha^2+\alpha + 2\rangle$ with the $2$-adic valuation. 
   For any extension of the $2$-adic valuation on $\Q$ to $K$, $x^2+x+2$ has one root of valuation $1$, which we can take to be $\alpha$, and one of valuation $0$, which is then $-1-\alpha$. Then $\overline{\alpha} = -1-\alpha$ extends to an automorphic involution on $K$. 
   We see that $\nu(\alpha)\neq \nu(\overline{\alpha})$. 
   Consider the matrix 
   \[
A = \begin{pmatrix}
     4 & 4+\alpha & 3-\alpha \\
 3-\alpha & 4 & 4+\alpha\\
 4+\alpha & 3-\alpha & 4 \\
\end{pmatrix}.
\]
Note that $\overline{A} = A^T$. The principal minors of $A$ are encoded by the coefficients of 
   \[
   f = \det({\rm diag}(x_1, x_2, x_3) +A) = 
x_1x_2x_3 +4(x_1x_2+x_1x_3+x_2x_3) +2(x_1+x_2+x_3) - 55.
   \]
Here $\nu(A_S)$ equals $0$ for $|S|\in \{0,3\}$, $2$ for $|S|=1$ and $1$ for $|S|=2$. The regular subdivision induced by the function $S\mapsto \nu(A_S)$ has an edge $\{(0,0,0),(1,1,1)\}$, showing that $p_S = \nu(A_S)$ is not a valuated $\Delta$-matroid. 
\end{example}

\subsection{Representability in low dimensions}

We next discuss representability of valuated $\Delta$-matroids in low dimensions.

\begin{proposition}
Every element $p\in\Delta$-${\rm Dr}(3)$ with $p_{\emptyset}=0$ is realizable as the valuation of the principal minors of a Hermitian matrix.
\end{proposition}
\begin{proof}
Suppose $p\in \Delta$-$\Dr(3)$ and $p_{\emptyset}=0$.  
By \Cref{lem:valuated-delta-matroid-3}, the minimum of 
$    
2p_\emptyset + 2p_{123}$, $2p_1 + 2p_{23}$, $2p_2 + 2p_{13}$, $2p_3 + 2p_{12}$, 
    $p_\emptyset + p_{12} + p_{13} + p_{23}$, and $p_1 + p_2 + p_3 + p_{123}$
    is achieved on at least two of its arguments, or on one of the last two arguments.

To obtain the result, we can work over the field of real Hahn series $F = \R[\![t^{\R}]\!]$, whose nonzero elements consist of formal power series $\gamma = \sum_{r\in \R}c_rt^r$ for which $c_r\in \R$ and ${\rm supp}(\gamma)=\{r: c_r\neq 0\}$ is well ordered. 
This is a real closed field whose value group is all of $\R$. See, e.g., \cite[(6.10)]{R92}.
The valuation of an element $\sum_{r\in \R}c_rt^r$ is $\min\{r: c_r\neq 0\}$. Such an element is positive (or negative) if the coefficient of its term with minimum valuation is positive (or negative, respectively). This is the fixed field of $K = \C[\![t^{\R}]\!]$ under the automorphic involution induced by usual complex conjugation on $\C$.

    By \cite{AV21}, the image of $3\times 3$ Hermitian matrices (over $\C[\![t^{\R}]\!]$) under the principal minor map is defined by the equation $a_{\emptyset}=1$ and the inequalities 
    $a_ia_j - a_{ij} \geq 0$, $a_{ik}a_{jk}-a_ka_{ijk}\geq 0$ for all distinct $i,j,k$, and ${\rm HypDet}({\bf a})\leq 0$ where 
    \begin{align*}
{\rm HypDet}({\bf a}) 
=& \ a_{\emptyset}^2 a_{123}^2 +  a_{1}^2 a_{23}^2 +a_{2}^2 a_{13}^2+a_{3}^2 a_{12}^2  - 2 a_{\emptyset} a_{1} a_{23} a_{123} - 2 a_{\emptyset} a_{2} a_{13} a_{123} -2 a_{\emptyset} a_{3} a_{12} a_{123}\\
 &-2 a_{1} a_{2} a_{13} a_{23} - 2 a_{1} a_{3} a_{12} a_{23} -2 a_{2} a_{3} a_{12} a_{13} + 4 a_{\emptyset}  a_{12} a_{13} a_{23}+ 4 a_{1} a_{2} a_{3}a_{123} .
 \end{align*}
Thus it suffices to show there exists $(a_S)_{S\subseteq 3}$ with $a_S\in \R[\![t^\R]\!]$ such that $\nu(a_S) = p_S$ for all $S$ and such that $a_{\emptyset} = 1$, $a_ia_j - a_{ij} \geq 0$, and $a_{ik}a_{jk}-a_ka_{ijk}\geq 0$ for all $i,j,k$, and ${\rm HypDet}({\bf a}) \leq 0$.
We construct this explicitly as follows. 
Define $a_i = -t^{p_i}$ and $a_{ij} = -t^{p_{ij}}$ for all $i,j$, and define $a_{\emptyset} = 1$ and $a_{123} = t^{p_{123}}$. Note that $\nu(a_S) = p_S$ for all $S$.
By our choice of signs, 
\[a_ia_j - a_{ij} = t^{p_i+p_j}+t^{p_{ij}} \geq 0  \ \text{ and } \ a_{ik}a_{jk} - a_k a_{ijk} = t^{p_{ik}+p_{jk}}+t^{p_k+p_{123}} \geq 0.\]

Finally, we claim that ${\rm HypDet}({\bf a})\leq 0$. 
Note that for our choice of ${\bf a}$, ${\rm HypDet}({\bf a})=$\begin{align*} &  t^{2p_{123}} + t^{2p_1 + 2p_{23}}+t^{2p_2 + 2p_{13}} +t^{2p_3 + 2p_{23}} - 2 t^{p_1+p_{23}+p_{123}}- 2 t^{p_2+p_{13}+p_{123}}- 2 t^{p_3+p_{12}+p_{123}}\\ &- 2 t^{p_1+p_2+p_{13}+p_{23}} - 2 t^{p_1+p_3+p_{12}+p_{23}} -2 t^{p_2+p_3+p_{12}+p_{13}} - 4 t^{p_{12}+p_{13}+p_{23}} - 4t^{p_{1}+p_{2}+p_{3}+p_{123}}. \end{align*}
Let $m$ be the minimum from 
\Cref{lem:valuated-delta-matroid-3}. 
Note that $m$ is also the minimum valuation of the terms appearing in the above expression for ${\rm HypDet}({\bf a})$.
Consider the set $\mathcal{S} = \{S\subset[3]:|S|\leq 1,  p_S+p_{S^C}=m\}$.
If one of the last two terms of the minimum appearing in \Cref{lem:valuated-delta-matroid-3} has valuation $m$, then the coefficient  $t^m$ in the expression for  ${\rm HypDet}({\bf a})$ above is at most $-4+|\mathcal{S}| - 2 \binom{|\mathcal{S}|}{2}<0$, showing that  ${\rm HypDet}({\bf a})<0$. 

If neither  of the last two terms of the minimum appearing in \Cref{lem:valuated-delta-matroid-3} has valuation $m$, then, since $p\in \Delta$-${\rm Dr}(3)$,  $|\mathcal{S}|\geq 2$ and the coefficient of $t^m$ in ${\rm HypDet}({\bf a})$ is 
$|\mathcal{S}| - 2 \binom{|\mathcal{S}|}{2}$. This is strictly negative when $|\mathcal{S}|\geq 3$, implying ${\rm HypDet}({\bf a})<0$. If $|\mathcal{S}|=2$, then the coefficient of $t^m$ in ${\rm HypDet}({\bf a})$ is zero. The only remaining terms with positive coefficients are $t^{2(p_{S}+p_{S^C})}$ for $S\subseteq [3]$ with  $S\not\in \mathcal{S}$. 
Let $m'=\min\{2(p_{S}+p_{S^C}): S\subseteq [3], |S|\leq 1, S\not\in \mathcal{S}\}$. 
Note that $-2t^{(m+m')/2}$ appears in the above expression for ${\rm HypDet}({\bf a})$ and has valuation strictly less than that of $t^{m'}$ and therefore less that of any term with positive coefficient. It follows that the leading term of ${\rm HypDet}({\bf a})$ has negative coefficient and thus ${\rm HypDet}({\bf a})<0$. 
\end{proof}

We note that for $n\geq 4$, there are valuated $\Delta$-matroids that are not representable as the valuations of principal minors of any matrix. This follows from the dimension counts in \Cref{sec:dim}.

\begin{corollary}
There exists $p\in \Delta$-${\rm Dr}(4)$ with $p_{\emptyset}=1$ that is not achieved as the valuations of principal minors of a $4\times 4$ matrix over $K$.
\end{corollary}

\begin{proof}
    By \cite{Stouffer}, over any field $K$, the image of the principal minors of an $n\times n$ matrix are contained in a variety of dimension $n^2-n+1$. By the structure theorem in tropical geometry \cite[Thm 3.3.5]{TropBook}, the image of this set under coordinate-wise valuation has dimension at most $n^2-n+1$. 
    In particular, for $n=4$, this has dimension $13$. 
    
    On the other hand, by \Cref{prop:DimDr4},  $\Delta$-${\rm Dr}(4)$ has codimension one in $\R^{16}$. It is not difficult to check that the set of $p\in \Delta$-${\rm Dr}(4)$ with $p_{\emptyset}=0$ therefore has dimension $14$. 
\end{proof}

\subsection{Extensions of Theorem~\ref{thm:HermValDelta} and conjectures}

By taking any valued field with the trivial involution $\overline{a} = a$, we obtain the following from \Cref{thm:HermValDelta}.

\begin{corollary} \label{cor:sym}
    The valuations of the principal minors of symmetric and skew-symmetric matrices over $K$ are valuated $\Delta$-matroids. 
\end{corollary}

In fact, the above result can be extended to matrices which are skew-symmetric plus a rank one symmetric matrix.

\begin{corollary} \label{cor:skewsymplusone}
    Let $A = B + C$ where $B \in K^{n \times n}$ is skew-symmetric, and $C \in K^{n \times n}$ is symmetric of rank one. Then the valuations of the principal minors of $A$ is a valuated $\Delta$-matroid.
\end{corollary}

\begin{proof}
    Let $K(t)$ be the transcendental field extension of $K$ by an element $t$. Invoking \cite[Theorem 2.2.1]{ValuedFields}, we can extend $\nu$ to a valuation on $K(t)$ by defining $\nu(\sum_i a_i t^i) = \min_i\{\nu(a_i)\}$ and $\nu(f/g)=\nu(f)-\nu(g)$. The field $K(t)$ has an automorphic involution $t\mapsto -t$ that preserves the valuation.

    Define $A' = tB + C$. This is a matrix over $K(t)$ which is Hermitian with respect to the above involution. We claim that $\nu(A_S) = \nu(A_S')$ for all $S \subset [n]$. Since the valuations of principal minors of $A'$ is a valuated $\Delta$-matroid by \Cref{thm:HermValDelta}, this will prove the desired result.
    
    We can write $C = \alpha v v^T$ where $\alpha \in K$ and $v \in K^n$.
    Let $B[S]$ denote the principal submatrix of $B$ with rows and columns indexed by $S$, so that $B_S=\det(B[S])$. Similarly, for the vector $v\in K^n$, let $v[S]$ denote the subvector consisting of with entries of $v$ indexed by $S$. 
    Then $A[S] = B[S]+ \alpha v[S]v[S]^T$ and $A'[S] = tB[S]+ \alpha v[S]v[S]^T$.
    Using the matrix determinant lemma, we 
    have
    \[   A_S  = B_S + \alpha v[S]^T {\rm adj}(B[S]) v[S] \ \ \text{ and } \ \
        A'_S  = t^{|S|}B_S + \alpha t^{|S|-1}v[S]^T {\rm adj}(B[S]) v[S]. \]
    where ${\rm adj}(M)$ is the adjugate matrix of $M$.
    If $|S|$ is odd, then $B_S= 0$ since $B[S]$ is skew-symmetric. In this case
    $A_S = \alpha v[S]^T {\rm adj}(B[S]) v[S]$  and 
    $A'_S = t^{|S|-1}A_S$, 
    showing that $\nu(A_S) = \nu(A'_S)$, as desired.
    If $|S|$ is even, then ${\rm adj}(B[S])$ is skew-symmetric, so $v[S]^T {\rm adj}(B[S]) v[S] = 0$. In this case
    $A_S  = B_S$ and $A'_S  = t^{|S|}B_S$, again showing that $\nu (A_S) = \nu(A'_S)$.
\end{proof}

We conjecture that we can extend Corollary~\ref{cor:skewsymplusone} as follows.

\begin{conjecture} \label{conj:skewHermplusone}
    Let $K$ be a valued field with an automorphic involution preserving the valuation. Let $A = B + C$ where $B \in K^{n \times n}$ is skew-Hermitian, and $C \in K^{n \times n}$ is Hermitian of rank one. Then the valuations of the principal minors of $A$ form a valuated $\Delta$-matroid.
\end{conjecture}

As discussed in the next section, this conjecture has implications for the representability of $\Delta$-matroids by isotropic subspaces of Hermitian spaces. We will prove the conjecture in the case where the residue field of $K$ does not have characteristic 2.

\begin{proposition} \label{prop:conjcase}
    Conjecture~\ref{conj:skewHermplusone} is true if the residue field of $K$ does not have characteristic 2.
\end{proposition}

We begin by proving the following.

\begin{lemma} \label{lem:notchar2}
    Let $K$ be a field with nonarchimedean valuation $\nu$ and automorphic involution $a \mapsto \overline{a}$ preserving the valuation. Assume the residue field of $K$ does not have characteristic 2. Let $a$, $b \in K$ such that $\overline{a} = a$ and $\overline{b} = -b$. Then $\nu(a+b) = \min(\nu(a),\nu(b))$.
\end{lemma}

\begin{proof}
    We only need to prove this when $\nu(a) = \nu(b)$.
    Since the residue field of $K$ does not have characteristic 2, we have $\nu(2) = 0$. Therefore,
    \[
    \nu(a) = \nu(2a) = \nu( (a+b) + (a-b) ) \ge \min ( \nu(a+b), \nu(a-b) ) = \nu(a+b)
    \]
    where the last equality holds since $\overline{a+b} = a-b$. Thus $\nu(a+b) \le \min(\nu(a),\nu(b))$ since $\nu(a) = \nu(b)$. Hence $\nu(a+b) = \min(\nu(a),\nu(b))$.
\end{proof}

\begin{proof}[Proof of Proposition~\ref{prop:conjcase}]
    Let $K(t)$ be the transcendental field extension of $K$ by an element $t$, and extend the valuation on $K$ to $K(t)$ as in the proof of Corollary~\ref{cor:skewsymplusone}. Equip $K(t)$ with the automorphic involution sending $t$ to $-t$ and $a$ to $\overline{a}$ for all $a \in K$. This involution preserves the valuation on $K(t)$.

    Define $A' = tB + C$. This is a matrix over $K(t)$ which is Hermitian with respect to the above involution. We claim that $\nu(A_S) = \nu(A_S')$ for all $S \subset [n]$, which would give the desired result by \Cref{thm:HermValDelta}.

    Let $F$ be the fixed field of the involution on $K$.
    We can write $C = \alpha v \overline{v}^T$ where $\alpha \in F$ and $v \in K^n$. As in the proof of Corollary~\ref{cor:skewsymplusone}, we have
    \[   A_S  = B_S + \alpha \overline{v}[S]^T {\rm adj}(B[S]) v[S] \ \ \text{ and } \ \
        A'_S  = t^{|S|}B_S + \alpha t^{|S|-1}\overline{v}[S]^T {\rm adj}(B[S]) v[S]. \]
    Let $G$ denote the set of all $a \in K$ such that $\overline{a} = -a$. If $|S|$ is odd, then since $B$ is skew-Hermitian, we have $B_S \in G$ and $\alpha \overline{v}[S]^T {\rm adj}(B[S]) v[S] \in F$. If $|S|$ is even, then we have $B_S \in F$ and $\alpha \overline{v}[S]^T {\rm adj}(B[S]) v[S] \in G$. Either way, by Lemma~\ref{lem:notchar2}, we have
    \[
    \nu(A_S) = \min( \nu(B_S), \nu(\alpha \overline{v}[S]^T {\rm adj}(B[S]) v[S]) ).
    \]
    By definition, this also the valuation of $A_S'$, completing the proof.
\end{proof}

When $\chara K = 2$, Conjecture~\ref{conj:skewHermplusone} is true because then $A$ is Hermitian. In this case we make another conjecture. Assume $\chara K = 2$. Let $F$ be the fixed field of the involution on $K$, and assume $F \neq K$. By standard Galois theory, we have $K = F[\omega]$ where $\omega^2 + \omega + \beta = 0$ for some $\beta \in F$, and $\overline{\omega} = \omega + 1$.

\begin{conjecture} \label{conj:char2plusone}
    Let $K = F[\omega]$ be as above. Let $A = B + \omega C$ where $B$, $C \in K^{n \times n}$ are Hermitian and $C$ has rank one. Then the valuations of the principal minors of $A$ form a valuated $\Delta$-matroid.
\end{conjecture}

\begin{proposition} \label{prop:conjcase2}
    Conjecture~\ref{conj:char2plusone} is true if the valuation is trivial.
\end{proposition}

\begin{proof}
    Assume the valuation $\nu$ is trivial. Then for all $a$, $b \in F$ we have $\nu(a + b \omega) = \min( \nu(a), \nu(b \omega) )$. Indeed, $a + b \omega = 0$ if and only if $a$ and $b \omega$ are both 0. The result then follows from an argument similar to the proof of Proposition~\ref{prop:conjcase}.
\end{proof}

\section{Representability via sesquilinear forms}

We now restate Theorem~\ref{thm:HermValDelta} in the language of sesquilinear forms and isotropic Grassmannians. This work was initiated by Gelfand and Serganova, who proved the below results for symplectic and quadratic forms on trivially-valued complex vector spaces. Their results were extended to arbitrary fields in \cite{CoxeterMatroids} and \cite{EFLS24}. The purpose of this section is to show that these results extend to non-trivial valuations and Hermitian forms.

We begin by recalling some polar geometry. We refer to \cite{Cameron} for further information.
Let $K$ be a field equipped with an automorphic involution $\iota : K \to K$, possibly the identity. We write $\overline{a}$ for $\iota(a)$. Let $V$ be a finite-dimensional vector space over $K$. A \emph{sesquilinear form} on $V$ is a biadditive function $b : V \times V \to K$ satisfying $b(\alpha x, \beta y) = \alpha \overline{\beta} b(x,y)$ for all $\alpha$, $\beta \in K$ and $x$, $y \in V$. 
A sesquilinear form is \emph{nondegenerate} if for all $x \in V \setminus \{0\}$ neither of the functions $b(x,\cdot)$ and $b(\cdot,x)$ are identically 0.
A sesquilinear form is \emph{reflexive} if $b(x,y) = 0$ implies $b(y,x) = 0$.
A sesquilinear form is \emph{Hermitian} if $b(x,y) = \overline{b(y,x)}$.

We have the following characterization of nondegenerate reflexive sesquilinear forms.

\begin{theorem}{\cite[Thm 3.3]{Cameron}}
Let $b$ be a nondegenerate reflexive sesquilinear form over a field with automorphic involution $\iota$. Then one of the following is true.
\begin{enumerate}
    \item $\iota$ is the identity, and $b$ is alternating ($b(x,x) = 0$).
    \item $\iota$ is the identity, and $b$ is symmetric ($b(x,y) = b(y,x)$).
    \item $\iota$ is not the identity, and $b$ is a scalar multiple of a Hermitian sesquilinear form.
\end{enumerate}
\end{theorem}

In Case 1 we call $b$ a \emph{symplectic form}. Case 2 is usually not dealt with in full generality due to issues when $\chara K = 2$. Instead, we introduce quadratic forms. A \emph{quadratic form} on a vector space $V$ over $K$ is a function $q : V \to K$ for which there exists a bilinear form $b$ on $V$ satisfying
\[
q( \alpha x + \beta y ) = \alpha^2 q(x) + \alpha \beta b(x,y) + \beta^2 q(y)
\]
for all $\alpha$, $\beta \in K$ and $x$, $y \in V$. We necessarily have
\[
b(x,y) = q(x + y) - q(x) - q(y)
\]
so $b$ is a symmetric bilinear form determined by $q$. We call $b$ the bilinear form associated to $q$. If $\chara K \neq 2$, then $q$ is also determined by $b$ as $q(x) = \frac{1}{2}b(x,x)$. However, if $\chara K = 2$ there can be multiple quadratic forms associated to a bilinear form. A quadratic form $q$ with associated bilinear form $b$ is \emph{nondegenerate} if there is no $x \in V \setminus \{0\}$ such that $b(x,\cdot) = 0$ is identically 0 and $q(x) = 0$. (If $\chara K = 2$, a nondegenerate quadratic form can have a degenerate associated bilinear form.)

From now on, we will use the term ``Hermitian'' to mean we are in Case 3, that is the involution $\iota$ is not the identity.

Let $V$ be a finite-dimensional vector space equipped with a (nondegenerate) symplectic form $b$, a nondegenerate quadratic form $q$ with associated bilinear form $b$, or a nondegenerate Hermitian form $b$. In what follows, we will make statements for the symplectic and Hermitian cases, followed in parentheses by the corresponding statement for the quadratic case if different. An \emph{isotropic subspace} of $V$ is a linear subspace $W \le V$ such that $b$ (respectively, $q$) vanishes on $W$. Witt's theorem implies that all maximal isotropic subspaces of $V$ have the same dimension, called the \emph{Witt index} of $V$. If the Witt index is $n$, then $n \le \dim V / 2$, and we can write
\[
V = H_1 \oplus \dots \oplus H_n \oplus W
\]
where the summands are pairwise orthogonal with respect to $b$, $W$ is anisotropic (that is, $b(x,x) \neq 0$ (respectively, $q(x) \neq 0$) for all $x \in W \setminus \{0\}$), and each $H_i$ has a basis $\{e_i,f_i\}$ with $b(e_i,e_i) = b(f_i,f_i) = 0$ ($q(e_i) = q(f_i) = 0$) and $b(e_i,f_i) = 1$ for all $i$. The space of all maximal isotropic subspaces of $V$ is called the \emph{maximal isotropic Grassmannian}.

From now on assume $V$ has Witt index $n = \lfloor \dim V / 2 \rfloor$. Let $\{e_1,\dots,e_n,f_1,\dots,f_n\} \cup G$ be a basis of $V$ where $e_i$ and $f_i$ are as described above and $G$ is a basis of $W$. If $\dim V$ is odd, then we let $G = \{g\}$ where $b(g,g) = \alpha \neq 0$ ($q(g) = \alpha \neq 0$). Any maximal isotropic subspace $L$ can be represented by a $n \times (\dim V)$ full-rank matrix where the rows of the matrix form a basis for $L$ in the coordinates $(e_1,\dots,e_n,f_1,\dots,f_n) \cup G$. Let $\nu : K \to \R \cup \{\infty\}$ be a nonarchimedian valuation with $\nu(a) = \nu(\overline{a})$.
For each $S \subseteq [n]$, let
\[
p_S(L) := \nu \left( \det M[\{e_i : i \in S\} \cup \{f_i : i \in [n] \setminus S\}] \right),
\]
that is, the valuation of the minor of $M$ using the columns corresponding to $e_i$ with $i \in S$ and $f_i$ with $i \in [n] \setminus S$. Our result is as follows.

\begin{theorem} \label{thm:isograssmannian}
    Let $V$ be as above. If $\iota$ is not the identity, assume $V$ has even dimension. Then for any maximal isotropic subspace $L$, the function $S \mapsto p_S(L)$ is a valuated $\Delta$-matroid.
\end{theorem}

For example, if $\nu$ is trivial, then $p_S(L)$ defines an ordinary $\Delta$-matroid. These $\Delta$-matroids in the cases where $V$ is quadratic of odd dimension, symplectic, or quadratic of even dimension are known in the literature as representable of type $B$, $C$, and $D$ respectively. For $V$ Hermitian of even dimension, the $\Delta$-matroids we obtain are, up to a symmetry of the hypercube, those obtained from Hermitian matrices as described in the previous section.

We conjecture that Theorem~\ref{thm:isograssmannian} can be extended to odd-dimensional Hermitian $V$.

\begin{conjecture} \label{conj:isograssmannianodd}
    The conclusion of Theorem~\ref{thm:isograssmannian} holds when $\iota$ is not the identity and $V$ has odd dimension.
\end{conjecture}

We do not have a full proof of this. See Proposition~\ref{prop:partialisoconj} for partial progress.
    
We split the proof of Theorem~\ref{thm:isograssmannian} into the symplectic, quadratic, and Hermitian cases. In all cases, we write $V = H_1 \oplus \dots \oplus H_n \oplus W$ as above.

\subsection*{Symplectic case}

In this case, we necessarily have $W = \{0\}$ and $\dim V = 2n$. Let $M$ be a matrix representing $L$. Note that the projection of $L$ onto any $H_i$ is nontrivial, because otherwise $L$ is contained in a space of smaller Witt index, contradicting the fact that $L$ is a maximal isotropic subspace. Hence, by swapping $e_i$ and $f_i$ for some set of $i$, the last $n$ columns of $M$ can be made linearly independent. Without loss of generality we may do this, because this corresponds to precomposing the function $p_S(L)$ with $S \mapsto S \Delta i$, and the new function is a $\Delta$-matroid if and only if the old one is. Performing row operations, we may assume $M$ is of the form $\begin{pmatrix} A & I \end{pmatrix}$, where $I$ is the $n \times n$ identity matrix. Then $b$ vanishes on $L$ if and only if $A$ is symmetric. Moreover, $p_S(L)$ is exactly $\nu(A_S)$. Hence, $p_S(L)$ is the valuation of the principal minors of a symmetric matrix $A$, which is a valuated $\Delta$-matroid by Theorem~\ref{cor:sym}.

\subsection*{Quadratic case}

Let $M$ be a matrix representing $L$. Assume first that $\dim V$ is even. Then in the coordinates of the chosen basis, we have
\[
q(x_1,\dots,x_n,y_1,\dots,y_n) = \sum_{i=1}^n x_i y_i.
\]
As above we may assume $M = \begin{pmatrix} A & I \end{pmatrix}$. Let $A_i$ denote the $i$-th column of $A$. Then $q$ vanishes on $L$ if and only if for all $u \in K^{n}$,
\[
    q(u^T M) = 0 \iff \sum_{i=1}^n (u^T A_i) u_i = 0 \iff u^T A u  = 0.
\]
The last equation is true for all $u$ if and only if $A$ is skew-symmetric. Hence $p_S(L)$ is the valuation of the principal minors of a skew-symmetric matrix $A$, which is a valuated $\Delta$-matroid by Theorem~\ref{cor:sym}.

Now suppose $\dim V$ is odd. In the coordinates of the chosen basis, we have
\[
q(x_1,\dots,x_n,y_1,\dots,y_n,z) = \sum_{i=1}^n x_i y_i + \alpha z^2.
\]
where $\alpha \neq 0$. We may assume $M = \begin{pmatrix} A & I & v \end{pmatrix}$ where $v$ is a column vector. Then $q$ vanishes on $L$ if and only if for all $u \in K^{n}$,
\[
    q(u^T M) = 0 \iff \sum_{i=1}^n (u^T A_i) u_i + \alpha (u^T v)^2 = 0 \iff u^T (A + \alpha vv^T) u  = 0
\]
The last equation is true for all $u$ if and only if $A + \alpha v v^T$ is skew-symmetric. Hence we can write $A = B - \alpha v v^T$ where $B$ is a skew-symmetric matrix. The result then follows from Corollary~\ref{cor:skewsymplusone}.

\subsection*{Hermitian case}

First assume $\dim V$ is even. Let $M$ be the matrix representing $L$. We may assume $M = \begin{pmatrix} A & I \end{pmatrix}$. Then $b$ vanishes on $L$ if and only if $A$ is skew-Hermitian. The result then follows from Theorem~\ref{cor:sym}. This concludes the proof of Theorem~\ref{thm:isograssmannian}.

\subsection{Partial results for odd-dimensional Hermitian}

We make the following partial progress toward proving Conjecture~\ref{conj:isograssmannianodd}.

\begin{proposition} \label{prop:partialisoconj}
    Conjecture~\ref{conj:isograssmannianodd} is true in the following cases.
    \begin{enumerate}
        \item The residue field of $K$ does not have characteristic 2.
        \item The valuation is trivial.
    \end{enumerate}
\end{proposition}

\begin{proof}

Recall that $V$ is an odd-dimensional vector space over $K$ with non-degenerate Hermitian form $b$. In the coordinates of the chosen basis, we have
\[
b( (x_1,\dots,x_n,y_1,\dots,y_n,z), (x_1',\dots,x_n',y_1',\dots,y_n',z') ) = \sum_{i=1}^n x_i \overline{y_i'} + \sum_{i=1}^n y_i \overline{x_i'} + \alpha z \overline{z'}
\]
for some $\alpha$ in the fixed field $F$.
Let $M$ be the matrix representing $L$. We may assume $M = \begin{pmatrix} A & I & v \end{pmatrix}$ where $v$ is a column vector. Then $b$ vanishes on $L$ if and only if $A + \overline{A^T} + \alpha v \overline{v^T} = 0$.

First assume $\chara \kk \neq 2$. It follows that $\chara K \neq 2$. Thus $b$ vanishes on $L$ if and only if $A + \frac{1}{2}\alpha v \overline{v^T}$ is skew-Hermitian. Hence $A = B - \frac{1}{2}\alpha v \overline{v^T}$ for some skew-Hermitian $B$, and the result follows by Proposition~\ref{prop:conjcase}.

Now assume the valuation $\nu$ is trivial. By the previous case, we only need to consider the case where $\chara K = 2$. As discussed before Conjecture~\ref{conj:char2plusone}, since $\iota$ is not the identity we have $K = F[\omega]$ where $\overline{\omega} = \omega + 1$. Then $b$ vanishes on $L$ if and only if $A + \omega \alpha v \overline{v^T}$ is Hermitian. Thus $A = B +  \omega \alpha v \overline{v^T}$ for some Hermitian matrix $B$. The result then follows from Proposition~\ref{prop:conjcase2}.
\end{proof}

\begin{remark}
    As seen from the above proof, Conjecture~\ref{conj:isograssmannianodd} is equivalent to Conjecture~\ref{conj:skewHermplusone} when $\chara K \neq 2$, and is equivalent to Conjecture~\ref{conj:char2plusone} when $\chara K = 2$.
\end{remark}

\bibliographystyle{plain}

\begin{thebibliography}{10}

\bibitem{AV21}
Abeer Al~Ahmadieh and Cynthia Vinzant.
\newblock Characterizing principal minors of symmetric matrices via
  determinantal multiaffine polynomials.
\newblock {\em J. Algebra}, 638:255--278, 2024.

\bibitem{AV24}
Abeer Al~Ahmadieh and Cynthia Vinzant.
\newblock Determinantal representations and the image of the principal minor
  map.
\newblock {\em Int. Math. Res. Not. IMRN}, (10):8930--8958, 2024.

\bibitem{CoxeterMatroids}
Alexandre~V. Borovik, I.~M. Gelfand, and Neil White.
\newblock {\em Coxeter matroids}, volume 216 of {\em Progress in Mathematics}.
\newblock Birkh\"{a}user Boston, Inc., Boston, MA, 2003.

\bibitem{Bouchet}
A.~Bouchet.
\newblock Representability of {$\triangle$}-matroids.
\newblock In {\em Combinatorics ({E}ger, 1987)}, volume~52 of {\em Colloq.
  Math. Soc. J\'anos Bolyai}, pages 167--182. North-Holland, Amsterdam, 1988.

\bibitem{B07}
Petter Br\"{a}nd\'{e}n.
\newblock Polynomials with the half-plane property and matroid theory.
\newblock {\em Adv. Math.}, 216(1):302--320, 2007.

\bibitem{Cameron}
Peter Cameron.
\newblock Notes on classical groups, 2000.

\bibitem{DW92}
Andreas~W.M. Dress and Walter Wenzel.
\newblock Valuated matroids.
\newblock {\em Adv. Math.}, 93:214--250, 1992.

\bibitem{ValuedFields}
Antonio~J. Engler and Alexander Prestel.
\newblock {\em Valued fields}.
\newblock Springer Monographs in Mathematics. Springer-Verlag, Berlin, 2005.

\bibitem{EFLS24}
Christopher Eur, Alex Fink, Matt Larson, and Hunter Spink.
\newblock Signed permutohedra, delta-matroids, and beyond.
\newblock {\em Proc. Lond. Math. Soc. (3)}, 128(3):Paper No. e12592, 54, 2024.

\bibitem{larson2023rank}
Matt Larson.
\newblock Rank functions and invariants of delta-matroids.
\newblock {\em Electron. J. Combin.}, 32(2):Paper No. 2.4, 17, 2025.

\bibitem{TropBook}
Diane Maclagan and Bernd Sturmfels.
\newblock {\em Introduction to tropical geometry}, volume 161 of {\em Graduate
  Studies in Mathematics}.
\newblock American Mathematical Society, Providence, RI, 2015.

\bibitem{R92}
Paulo Ribenboim.
\newblock Fields: algebraically closed and others.
\newblock {\em Manuscripta Math.}, 75(2):115--150, 1992.

\bibitem{Felipe}
Felipe Rinc\'{o}n.
\newblock Isotropical linear spaces and valuated {D}elta-matroids.
\newblock {\em J. Combin. Theory Ser. A}, 119(1):14--32, 2012.

\bibitem{speyer2004tropical}
David~E. Speyer.
\newblock Tropical linear spaces.
\newblock {\em SIAM J. Discrete Math.}, 22(4):1527--1558, 2008.

\bibitem{Stouffer}
E.~B. Stouffer.
\newblock On the independence of principal minors of determinants.
\newblock {\em Trans. Amer. Math. Soc.}, 26(3):356--368, 1924.

\bibitem{Wenzel}
Walter Wenzel.
\newblock Pfaffian forms and $\delta$-matroids.
\newblock {\em Discrete Math.}, 115(1--3):253--266, 1993.

\end{thebibliography}

\end{document}